\documentclass{article}

\usepackage{amsmath}
\usepackage{amssymb, amsthm, url}
\usepackage[all]{xy}
\usepackage{graphicx}

\newtheorem*{thm}{Theorem}

\newtheorem{prop}{Proposition}
\newtheorem{lem}{Lemma}
\newtheorem{rem}{Remark}

\newtheorem{defin}{Definition}
\newtheorem{question}{Question}

\newtheorem{example}{Example}

\title{On the minimality of semigroup actions 
on the interval which are $C^1$-close to the identity}

\author{Katsutoshi Shinohara \thanks{Universit\'{e} de Bourgogne.}}

\date{}

\begin{document}

\maketitle

\begin{abstract}
We consider semigroup actions on 
the interval generated by two attracting maps. 
It is known that if the generators are sufficiently $C^2$-close
to the identity, then the minimal set coincides with the whole interval.
In this article, we give a counterexample to this result 
under $C^1$-topology.

{ \smallskip

  }
\end{abstract}

\section{Introduction}

In dynamical systems, the study of semigroup actions,
in other words, iterated function systems (IFSs), 
are interesting subjects not only in themselves, but also
from the viewpoint of study of systems generated by a single map.
For example, in the study of  properties of 
dynamically defined invariant sets, 
such as the unstable manifold of a hyperbolic set,
a difficulty comes from the point that they often exhibit 
complicated fractal structures.
A natural and powerful way to analyze  them is to reduce the original
system to IFSs, which enables us to command the 
deep theories developed in fractal geometry (see \cite{PT}).

There is another motivation of this kind. It comes from
their relationship with partially hyperbolic systems, 
which also attracts the attention of recent research. 
One typical example of such systems is given 
by skew product systems of IFSs (see \cite{DG, IlN}).  
Thus study of IFSs are expected to contribute to the 
study of partially hyperbolic systems. 
It is true that skew product systems are ``too clean" compared to 
general partially hyperbolic systems: The regularity 
of the holonomies among fibers in general systems is very low, 
while in skew product systems it is given by the identity map. 
Nonetheless, the study of IFSs from this viewpoint would contribute
to the study of partially hyperbolic systems to offer 
a ``first step" of investigations of such kind.

There are many researches to this direction, 
in particular under the setting where the ambient space is a one dimensional manifold
(see for example \cite{ABY, BR1, BR2, DGR, KV, IlN, BR}). 
Very recently, under some generic assumptions, 
``spectral decomposition theorems for IFSs" 
were established by several researchers
(see for example \cite{BR, KV}).
Thus the next natural question would be to investigate the 
properties of each ``basic pieces'' that appear in such 
decompositions. 
One interesting thing in relation to this decomposition 
is that it is reported that if the systems has 
good regularity, then there are strange restrictions on them. 
The aim of this article is to obtain a better understanding about such restrictions. 

To explain the strange restrictions, we restrict our attention to a special setting.
Let $I := [0,1]$ and $f, g : I \to I$ be continuous maps satisfying the following conditions:
\begin{itemize}
\item $f(0) = 0$, $g(1)= 1$.
\item For $x \neq 0$,  $ 0<f(x)<x$ and for $x \neq 1$,  $x < g(x) < 1$.
\item $f$, $g$ are orientation preserving homeomorphisms on their images.
\end{itemize}
We denote the set of pairs of $(f, g)$ satisfying conditions above 
by $\mathcal{C}$. Furthermore, we denote the set of $(f, g)$ such that
$f, g$ are $C^1$-diffeomorphisms on their images satisfying
$ f'(0), g'(1) \in (0, 1)$ by $\mathcal{D}$.
Under this setting, we are interested in the property of {\it forward minimal sets}
of semigroup actions on $I$ generated by $f$ and $g$. 
A non-empty set $M \subset I$ is called a \emph{(forward) minimal set} if 
for every point in $M$, the forward orbit is dense in $M$.
We say that a semigroup action is {\it minimal} if the ambient space itself is a 
minimal set, in other words, every point has a dense orbit.
The reason why we assigned $f, g$ the above conditions is that 
this is the ``simplest'' case for the study of minimal set. 
The importance of the minimal set comes from the 
fact that they corresponds to homoclinic classes in the 
ambient space (see Section \ref{s.basic} for detail).

In \cite{BR}, 
the following property proved by Duminy is discussed (see also section 3.3 of \cite{N}). 
\begin{prop}[Duminy]\label{p.dum}
If $(f, g)$ are sufficiently $C^2$-close to the identity, 
the forward minimal set of $(f, g)$-action is equal to $I$.
In other words, every point has a dense orbit.
\end{prop}

The aim of this paper is 
to consider the similar problem under the $C^1$-topology.
The conclusion is that in $C^1$ case, we can construct a counterexample. 
More precisely, we prove the following.
\begin{thm}
There exists a sequence $(f_n, g_n) \subset \mathcal{D}$ 
converging to $(\mathrm{id}, \mathrm{id})$ in the $C^1$-topology 
such that $(f_n, g_n)$ has minimal set that does not equal to $I$.
\end{thm}

Let us discuss the importance of the closeness 
to the identity map in the assumption of the Theorem. 
The first importance is that it examines a shared intuition that 
``a groups generated by elements close to the identity have nice structures"
(for example see section 3.3 of \cite{N}). The result of Duminy says that 
such an intuition is valid under our context (for there is a uniformity on
the shape of minimal set.)
Meanwhile, our Theorem says that the $C^1$-closeness 
is not enough to guarantee the intuition above with respect to the 
shape of the minimal set. 

The second importance is from their 
relationship with partially hyperbolic systems. 
Consider an IFS whose skew product is not
uniformly hyperbolic to the center direction. 
Then we expect that there would be
some regions where the central direction has 
Lyapnov exponents with small absolute values. 
The condition that the dynamics 
are close to the identity are naturally satisfied in such regions. 
Thus, we can interpret Proposition \ref{p.dum} to mean that 
the homoclinic classes in such bifurcation regions cannot be small 
(the projection always become intervals). Indeed, there are 
several results which seems to reflect this principle. For 
example, in the proof of \cite{Tsuji}, one important step 
to prove the finitude of SRB measures is showing
that ``the support of ergodic measures which exhibits weak hyperbolicity to
central direction is always large," see the introduction of \cite{Tsuji}.
From this viewpoint, our Theorem can be regarded
as a negative circumstantial evidence for the validity of corresponding 
result under the $C^1$-topology. See also \cite{VY}. 

Let us give some comments on the proof of Theorem. 
The above Theorem gives a sequence of examples of non-minimal 
semigroup actions. Thus it is natural to ask the finer information of the
minimal sets, such as Hausdorff dimension or topology of them.
The only thing our proof tells is that they do not contain several intervals.
To calculate Hausdorff dimension or decide a topology, it seems 
that our argumen is not enough. 
So we do not touch this problem in this article.

The proof is done by mixture of elementary combinatorial 
and analytic arguments. In the earlier stage of this research
the author found some inspiration from papers such as \cite{M1, M2, U},
but most of the arguments in this paper are independent from them.
Since he cannot find adequate reference for the basic facts used in this article, 
in section 2 we prepare basic definitions and furnish the proof of basic facts.

\quad

{\bf Contents} \quad In section 2, 
we give the definitions of basic notions and
give proofs of basic facts. 
In section 3, we explain the strategy 
of our construction. The proof consists of two parts: Construction of 
``three kind of parts" of dynamical systems and assembly of them. 
Assuming the existence of such parts, we also 
give the argument of assembly in section 3. 
In section 4, we prepare some combinatorial arguments. 
In section 5 and 6, we prove the existence of two parts
called runway and connector. 
In section 7, 8, and 9, we prove the existence 
of the quantum leap, which completes the whole argument.

\quad

{\bf Acknowledgements} \quad
The author is thankful to Artem Raibekas, who spared 
his time for discussions with the author at the earlier stage of this research. 
He is also thankful to the people in PUC-Rio (Brazil). 
He gave sequential seminars on this topic there, and
that experience definitely contributed to improve the manuscript.
In particular, he thanks Jairo Bochi, who piloted me 
to this subject and listened to earlier, immature version of 
the argument with great patience and insightful comments
(his comment on Proposition \ref{p.goodrat} 
enabled the author to improve the proof considerably).   
Most part of this paper was prepared during his stay at PUC-Rio 
as a post doctoral researcher of CNPq (PDJ). 

\section{Preliminaries}\label{s.basic}
\subsection{Minimal set for $\mathcal{C}$}
In this section, we give basic definitions and collect some results. 

Let $X$ be a topological space and $f, g:X \to X$ be maps. 
By $\langle f, g \rangle_{+}$, we denote the semigroup generated by $f, g$,
in other words, $\langle f, g \rangle_{+}$ is the set of all possible 
finite concatenations of $f$ and $g$. 
This semigroup acts to $X$ in the natural way.  
For $x \in X$, we put 
$\mathcal{O}_{+}(x) := \{ \phi(x) \mid \phi \in \langle f, g \rangle_{+}\}$ 
and call it \emph{orbit} of $x$.
A non-empty set $M \subset X$ is called a \emph{minimal set}
if for every $x \in M$,  $\overline{\mathcal{O}_{+}(x)} =M$. 
The semigroup action $\langle f, g \rangle_{+}$ is called \emph{minimal}
if $X$ is the minimal set. 
In other words, if for every $x \in X$, $\overline{\mathcal{O}_{+}(x)} =X$. 

Let us consider the case where $X =I$ and
 $(f, g) \in \mathcal{C}$ (see Introduction for definition).
In this setting, we have the following.
\begin{lem}\label{l.unimi}
There exists a unique minimal set $M$ in $I$. Furthermore, 
we have $M = \overline{\mathcal{O}_{+}(0)} = \overline{\mathcal{O}_{+}(1)}$
\end{lem}

\begin{proof}
First, by the definition of $f$, 
it is easy to observe that for every $x \in I$, $f^n(x) \to 0$ as $n \to \infty$.
This implies that for every $x \in I$, 
we have $\overline{\mathcal{O}_{+}(0)} \subset \overline{\mathcal{O}_{+}(x)}$
(since for every $y \in \mathcal{O}_{+}(0)$ 
we can find a sequence in $\mathcal{O}_{+}(x)$ that converges to $y$,
by taking the image of the sequence converges to $0$).

We show that $\overline{\mathcal{O}_{+}(0)}$ is a minimal set. 
To see this, we fix $y \in \overline{\mathcal{O}_{+}(0)}$. 
By the above argument, we know 
$\overline{\mathcal{O}_{+}(0)} \subset 
\overline{\mathcal{O}_{+}(y)}$.
Let us see the inclusion of the opposite direction. 
Since $y \in \overline{\mathcal{O}_{+}(0)}$, the orbit of $0$ 
approaches arbitrarily close to $y$. 
By taking the image of this sequence, 
we can see that every orbit of $y$ is approached by the orbit of $0$
arbitrarily near. Thus we have 
$\mathcal{O}_{+}(y) \subset \overline{\mathcal{O}_{+}(0)}$.
By taking closure, we get the conclusion. 

Let us see the uniqueness. Suppose $M'$ is a minimal set. 
Fix $x \in M'$. We have 
$M' =\overline{\mathcal{O}_{+}(x)}$ by definition. 
Since $0 \in \overline{\mathcal{O}_{+}(x)}$, we have $0 \in M'$.
Now the minimality of $M'$ 
implies $M' =\overline{\mathcal{O}_{+}(0)}$.

Finally, with the symmetry of the conditions on $0$ and $1$, 
we conclude
$M' = \overline{\mathcal{O}_{+}(0)} = \overline{\mathcal{O}_{+}(1)}$.
\end{proof}

\begin{rem}
\begin{enumerate}
\item Existence of the (forward) minimal set is a quite general result.
One can always obtain it, for example, if $X$ is a compact Hausdorff space 
and $f, g$ are continuous.

\item Lemma \ref{l.unimi} explains why $\mathcal{C}$ is a appropriate space 
to investigate the property of minimal set.
Because of the uniqueness, there is no need to worry about 
the notion of ``continuation," in other words, there is no 
collision, division of minimal set in our context. 

\item Lemma \ref{l.unimi} also explains why minimal sets are important. 
If we take a skew product, then $\overline{\mathcal{O}_{+}(0)} = \overline{\mathcal{O}_{+}(1)}$ corresponds to the homoclinic class 
of fixed point corresponding to $0$ (which coincides with that of  $1$). 
\end{enumerate}
\end{rem}

\subsection{A sufficient condition for non-minimality}
We prepare some notation. 
Let $X$ be a connected one-dimensional manifold
(that is, $\mathbb{R}$, $\mathbb{R}_{\geq 0}$
$S^1$ or $I$). 
A non-empty subset $Y \subset X$ is called \emph{region}
if it satisfies the following:
\begin{itemize}
\item $Y$ is disjoint union of closed intervals with non-empty interior: $ Y := \coprod I_i$.
\item $Y$ is locally finite: For every compact interval $Z \subset X$, 
the number of connected component of $Y \cap Z$ is finite.
\end{itemize}

Note that if $X$ is compact, 
then the number of connected components 
of a region in $X$ is finite.

We want to give a sufficient condition for non-minimality. 
To describe it, we need a definition.

\begin{defin}
Let $X$ be a one dimensional manifold and $f, g : X \to X$.
A non-empty region $\mathcal{K}  \subset X $ is said to be 
a \emph{hiding region} for $(f, g)$ if the following holds:
\begin{itemize}
\item $\mathcal{K} \neq X$.
\item $\mathcal{K} \cap f(X) \subset f(\mathcal{K})$.
\item $\mathcal{K} \cap g(X) \subset g(\mathcal{K})$.
\end{itemize}

We call the second and the third conditions \emph{hiding property}.
Furthermore, we say that a hiding region is a \emph{strong hiding region}
if we change the second and the third conditions to 
 $\mathcal{K} \cap f(X) \subset \mathrm{int}(f(\mathcal{K}))$ and
 $\mathcal{K} \cap g(X) \subset \mathrm{int}(g(\mathcal{K}))$.
 \end{defin}

Roughly speaking, the existence of hiding region
implies the non-minimality of $\langle f, g \rangle_{+}$.
To see this, we collect some basic results in the case
$X = I$ and $(f, g) \in \mathcal{C}$.
\begin{lem}\label{l.non0}
Let us consider the case $X = I$.
Suppose $(f, g) \in \mathcal{C}$ and 
$\mathcal{K} \subset I$ is a hiding region for $(f, g)$, 
then $0, 1  \not\in  \mathcal{K}$.
 \end{lem}

\begin{proof}
Suppose $0 \in \mathcal{K}$. 
Then, there exists a connected component $K_i$ 
of $\mathcal{K}$
with non-empty interior containing $0$. Especially, $K_i$ contains 
one of the fundamental domain of $f$. This means, by taking 
the backward image of $K_i$ under $f$, we see whole $I$ 
is contained in $\mathcal{K}$, but this contradicts to the assumption 
that $\mathcal{K}$ is a proper subset of $I$.
The proof $1 \not\in K$ can be done similarly. 
\end{proof}

\begin{lem}\label{l.nonf0}
If $\mathcal{K} \subset I$ is a hiding region for 
$(f, g) \in \mathcal{C}$. Then 
$f(1), g(0) \not\in \mathcal{K} $.
 \end{lem}

\begin{proof}
Suppose $f(1) \in \mathcal{K}$. 
Since $\{f(1)\} \cap f(I) \neq \emptyset$ 
and $f$ is a homeomorphism on 
its image,  we have $1 \in \mathcal{K} $, 
but this contradicts to Lemma \ref{l.non0}.
The proof $g(0) \in \mathcal{K}$ can be done similarly. 
So we omit it. 
\end{proof}

\begin{rem}\label{r.dicoh}
Lemma \ref{l.nonf0} above imply the following: If $\mathcal{K}$ is 
a hiding region for $(f, g)$, each connected component $K_i$ 
of $\mathcal{K}$ has the following dichotomy;
$K_i \subset f(I)$, otherwise $K_i \subset I \setminus f(I)$.
\end{rem}

The following gives us one sufficient condition:
\begin{prop}\label{p.nonmin}
Suppose $(f, g) \in \mathcal{C}$ has a hiding region $K \subset I$.
Then, the (unique) minimal set $M$ for $\langle f, g \rangle_{+}$ 
is not equal to the whole interval $I$.
\end{prop}

\begin{proof}
Indeed, we can show $M \cap \mathrm{int}(K) =\emptyset$.
To see this, we investigate the orbit of $0$ under $\langle f, g \rangle_{+}$.
We show that $\mathcal{O}_{+}(0) \subset I \setminus K$ by the induction.
By Lemma \ref{l.non0} we know that $0 \not\in \mathcal{K} $.
Suppose for every $w_i \in \langle f, g \rangle_{+} $ of length 
less than or equal to$k$,
we proved that $w_i(0) \not\in \mathcal{K}$.
Suppose that there exists $W \in \langle f, g \rangle_{+} $ of length
$k+1$ such that $W(0) \in \mathcal{K}$. Especially, 
there exists $w \in \langle f, g \rangle_{+}$ of length $k$
such that $W= f \circ w$ or $W= g \circ w$ holds. 

Let us consider the case where $W=f \circ w$ holds.
Then, we take a connected component $K_i$ such that $W(0) \in K_i$. There are two 
possibilities (see Remark \ref{r.dicoh}): $K_i \subset I \setminus f(I)$ or $K_i \subset f(I)$.
The first case cannot happen, since $K_i$ contains a point that is the image of $f$.
In the latter case, there exists a connected component 
$K_j$ such that $K_i \subset f(K_j)$.
This means $w(0) \in K_j$, but this is a contradiction.
The other case can be treated similarly. 
\end{proof}

\subsection{Some examples}

In this section, we see some examples of minimal sets for
IFSs in $\mathcal{D}$. 

\begin{example}\normalfont
Let us take $(f, g) \in \mathcal{D}$ such that there exists $C$
with $0<C<1$ such that 
 $f'$, $g'$ are uniformly less 
than $C$. In this case, the semigroup action 
$\langle f, g \rangle_{+} $ on $I$ is minimal. 
Indeed, take $x \in I \setminus \{0, 1\}$ and take an arbitrarily small
interval $J$ that contains $x$. We show that 
$\mathcal{O}_{+}(0) \cap J  \neq \emptyset$.
To see this, let us consider the backward image of $J$. 
First, if $J$ contains $0$ or $1$, then it is OK. 
Similarly, if $J$ contains $f(1)$ or $g(0)$ then we can have above. 

So, let us assume $J \cap \{0, 1, f(1), g(0) \} = \emptyset$.
In this case, we can take the image of whole $J$ 
under $f^{-1}$ or $g^{-1}$.
Indeed, as long as the backward image is  disjoint from 
$\{0, 1, f(1), g(0) \}$,
we can take the backward image. 
Since $f, g$ are uniformly contracting, 
$f^{-1}$, $g^{-1}$ are uniformly expanding. 
Thus under the backward iteration,
the length of $J$ grows exponentially. This means that 
under the finite backward iteration, $J$ must have non-empty 
intersection with  $\{f(1), g(0) \}$, which implies 
$\mathcal{O}_{+}(0) \cap J  \neq \emptyset$.
\end{example}

Let us see the example where the action is \emph{not} minimal.

\begin{example}\normalfont
We consider the IFSs on $I$ generated  by 
$f_{\varepsilon}(x) := (1/2 + \varepsilon)x$ and 
$g_{\varepsilon}(x) :=(1/2 + \varepsilon)(x-1) +1 $.
When $\varepsilon =0$, $g_0 \circ f_0$ has a unique contracting fixed point 
$p_0 := 2/3$. For small $\varepsilon$, we can define the continuation 
and  denote it by $p_\varepsilon$, 
and put  $q_{\varepsilon} :=f_{\varepsilon}(p_\varepsilon)$ 
(it is not difficult to write down explicit formulas of $p_{\varepsilon}$ 
and $q_{\varepsilon}$).

Then, we modify $f_{\varepsilon}$, $g_{\varepsilon}$ 
into $F_{\varepsilon}$, $G_{\varepsilon}$ as follows:

\begin{itemize}
\item For  $f_{\varepsilon}$.
First we fix a small neighborhood $I_{p_{\varepsilon}}$ of $p_{\varepsilon}$
(For the smallness, $I_{p_{\varepsilon}} \subset (0, 1/2 -\varepsilon) $ is enough). 
Then we modify $f_{\varepsilon}$ to $F_{\varepsilon}$ 
such that  
$F_{\varepsilon}$ satisfies $F'_{\varepsilon} >1$ around $p_{\varepsilon}$, 
keeping $F_{\varepsilon} =f_{\varepsilon}$ outside $I_{\varepsilon}$.

\item For  $g_{\varepsilon}$. We perform similar modifications to $g_{\varepsilon}$.
First we fix a small neighborhood $I_{q_{\varepsilon}}$ of $q_{\varepsilon}$
(For the smallness, $I_{q_{\varepsilon}} \subset ( 1/2 +\varepsilon, 1) $ is enough). 
Then we modify $g_{\varepsilon}$ to $G_{\varepsilon}$ 
such that  
$G_{\varepsilon}$ satisfies $G'_{\varepsilon} >1$ around $q_{\varepsilon}$, 
keeping $G_{\varepsilon} =g_{\varepsilon}$ outside $I_{q_{\varepsilon}}$.
\end{itemize}

Then, we fix an interval $J_{p_\varepsilon}$, $J_{q_\varepsilon}$ 
of the same length contained in $I$ centered 
at $p_{\varepsilon}$, and  $q_{\varepsilon}$. 
Then we can check $(F_{\varepsilon}, G_{\varepsilon})$ satisfies 
the condition in Proposition \ref{p.nonmin} with hiding region 
$\mathcal{J} := J_{p_{\varepsilon}} \coprod J_{q_{\varepsilon}}$. 
Thus this action is not minimal.
\end{example}

\begin{rem}
In this construction, the essential point is to obtain a point whose 
orbit under $f^{-1}$ and $g^{-1}$ is finite. So, to answer Question 1, 
it is enough to find pair $(f, g)$ such that $(f^{-1}, g^{-1})$ has 
finite orbit arbitrarily $C^1$-close to the identity.
However, it seems that this problem is very hard to solve,
so we do not pursue it in this article.
\end{rem}

\subsection{Robustness of non-minimality}

Suppose $\mathcal{K}$ is a hiding region for $(f, g) \in \mathcal{C}$. 
Then we know the action is not minimal. In this section, we consider 
if this properties are \emph{robust} or not.
The content of this section has no direct logical link for the proof of 
the Theorem.

\begin{rem}
For a region $\mathcal{K}$, being a strong hiding region 
is a $C^0$-robust property in the following sense:
Suppose $\mathcal{K}$ is a strong hiding region for $(f, g) \in \mathcal{C}$,
then there exists $\varepsilon >0$ such that
for every $(\tilde{f}, \tilde{g}) \in \mathcal{C}$ satisfying 
$d_{C^0} (f, \tilde{f}), d_{C^0} (g, \tilde{g}) < \varepsilon$, then 
$\mathcal{K}$ is also a hiding region for $(\tilde{f}, \tilde{g})$.
This comes from the fact that being a hiding region can be decided 
only from the information of the 
boundary points (which is finite) of $\mathcal{K}$.
\end{rem}

The following lemma says that the difference between 
hiding regions and strong hiding regions are ignorable.

\begin{lem}
Suppose $(f, g)\in \mathcal{D}$ has a hiding region $\mathcal{K}$.
Then, there exists $(F, G)\in \mathcal{D}$ $C^1$-arbitrarily close 
to $(f, g)$ such that $\mathcal{K}$ is a strictly hiding region. 
\end{lem}

\begin{proof}
Suppose $(f, g)$ has a hiding region $\mathcal{K}$ which 
is not strictly hiding. some of the boundary of the interval 
goes to the boundary of the others. Note that the number of 
the boundaries are finite: $\mathcal{K}$ consists of finite 
number of connected components. 
Thus, by modifying $f$ and $g$ around such points 
to cover $\mathcal{K}$ in its interior by a bump function
and keep intact 
around other points, we can obtain desired maps. 
\end{proof}

\begin{rem}
In fact, if $(f, g)$ are $C^r$ maps $(0 \leq r \leq +\infty, \omega)$, 
then $(F, G)$ can be taken $C^r$-arbitrarily close to $(f, g)$.
\end{rem}

\subsection{Another $C^1$-distance on $\mathcal{D}$}
We clarify the distance which 
we use to discuss the convergence of dynamical systems.
\begin{defin}
Let $f: I \to I$ be a map. We define 
$d_{C^0}(f) := \sup_{x \in I} |f(x) -x|$
(this gives a finite non-negative value if $f$ is bonded, 
especially, if $f$ is continuous).  
Furthermore, if $f$ is $C^1$, then we define
$d'_{C^1}(f) := \sup_{x \in I} |f'(x) -1|$ 
and $d_{C^1}(f) := d_{C^0}(f) + d'_{C^1}(f)$.

We say that $(f_n, g_n) \subset \mathcal{D}$ 
converges to $(\mathrm{id}, \mathrm{id})$ if 
$d_{C^1}(f_n), d_{C^1}(g_n) \to 0$ as $n \to +\infty$.
\end{defin}

This is the usual definition of $C^1$-convergence. 
In our context, we have a simpler (hence easy to handle) condition which
guarantees the convergence. 

We start from a simple observation:
\begin{lem}
Let $f: I \to I$ be a $C^1$-map with $f(0) = 0$.
Then we have $d_{C^0}(f) \leq d'_{C^1}(f)$.
\end{lem}
\begin{proof}
By definition of $f$, we have
$(1- d'_{C^1}(f) )x \leq f(x)  \leq (d'_{C^1}(f)+1 ) x$ for $x \in I$. 
Thus we have $|f(x) - x| \leq d'_{C^1}(f)x$. Take the supremum of each side. 
\end{proof}

As a conclusion, we have the following:
\begin{prop}\label{p.boot}
If a sequence of $C^1$ maps  
$(f_n)$ on $I$ satisfies $f_n(0) =0$ and $d'_{C^1}(f_n) \to 0$, then 
$d_{C^1}(f_n) \to 0$. 
\end{prop}
In other words, in our case we can consider that 
$d'_{C^1}$ defines a distance from the identity map.

Note that similar results hold for $g$ satisfying $g(1)=1$.

\subsection{Reduction to PL-case}

For the actual construction of non-minimal semigroup actions, 
we work on the space of PL-maps.
This subsection gives a translation of $C^1$-problem into PL-one.
In this article, a PL-map means piecewise linear map with  finite
non-differentiable points.

\begin{defin}
We say $(f, g) \in \mathcal{P}$ if $(f, g) \in \mathcal{C}$ and 
$f, g$ are PL-maps.
\end{defin}

We define a metric on this space.

\begin{defin}
Let $f$ be a PL-map, we denote the set of non-differentiable points 
by $c(f)$.
For a PL-map $f$ on $I$, we define 
$\mu(f) := \max_{x \in I \setminus c(f)} |f'(x)-1| $.
\end{defin}

The function $\mu(f)$ plays a similar role on $\mathcal{P}$
as $d'_{C^1}$ played on $\mathcal{D}$.
In this setting, we ask the following question.

\begin{question}\label{q.PL}
Is there $(f_n, g_n) \subset \mathcal{P}$ such that 
$\mu (f_n)$, $\mu (g_n) \to 0$ as $n \to +\infty$ and 
each $(f_n, g_n)$ has hiding region?
\end{question}

The point is that this question is equivalent to 
the previous one:
\begin{prop}\label{p.reduPL}
If we can solve Question \ref{q.PL}, then 
we can construct the sequence in Theorem from that solution. 
\end{prop}

Roughly speaking, what we need to prove 
is that we can remove  "corners" of PL-maps by arbitrarily small 
modifications. For the proof of the previous proposition, 
we prepare two lemmas.

\begin{lem}\label{l.plgene}
Let $f$ be a PL-map on $I$ and $X := \{ x_i \}$ be a 
set of finite points in $I$.
Then given neighborhood $U$ of $X$ there exists a PL-map $g$ satisfying all 
the following conditions:
(i) $\mu(f) = \mu(g)$,
(ii) $f(x_i) = g(x_i)$,
(iii) $f = g$ outside $U$, and
(iv) $c(g) \cap X = \emptyset$.
\end{lem}
\begin{proof}
If $c(f) \cap X = \emptyset$, then we just take $g = f$.
If not, for each $x_i \in c(f)$, we modify $f$ so 
that modifying the position 
of corner keeping $\mu(f)$. 
Giving such a modification is easily.
\end{proof}

The following lemma says that we can remove 
the corner keeping the most of characteristics of dynamics.

\begin{lem}\label{l.smoothing}
Let $f$ be a PL-map on $I$ which is homeomorphism on its image. 
Then given neighborhood $U$ of $c(f)$,
there exists a $C^1$-map $g$ on $I$ satisfying the following:
(i) $f = g$ outside $U$. (ii) $d'_{C^1}(g) = \mu (f)$.
\end{lem}

\begin{proof}
We do not give an explicit construction of $g$, 
but just describe an idea of the proof.
Look at the graph of $f'$. It is discontinuous. 
Then interpolate the graph to be continuous. 
Now, with some careful choice the
indefinite integral of interpolated function
gives the desired $C^1$-map.
\end{proof}

Now we give the proof of Proposition \ref{p.reduPL}.

\begin{proof}
Let $(f_n, g_n)\subset \mathcal{P}$ 
be the solution of Question \ref{q.PL} 
with hiding region $K_n$. 
First, we see that we can assume 
$c(f_n), c(g_n) \cap \partial (K_n) = \emptyset$,
since if not by applying Lemma \ref{l.plgene} 
we can change the position of $c(f)$ keeping 
the hiding region without changing $\mu(f)$ and $\mu(g)$.
Then, under the 
assumption $c(f_n), c(g_n) \cap \partial (K_n) = \emptyset$,
we apply Lemma \ref{l.smoothing} to each $ (f_n, g_n)$,
letting $U_n$ so small that $U_n$ 
does not touch the boundary of $K_n$.
Now the resulted sequence of $C^1$-maps 
has $K_n$ being hiding region, 
which converges to the identity in the $C^1$-topology 
by Proposition \ref{p.boot}.
\end{proof}

\section{Strategy for construction: Localization}

\subsection{Strategy}
In this subsection, we explain the rough idea of the proof of
the Theorem. 

The main difficulty of the construction comes from the non-abelian behavior 
of maps. In general, $f$ and $g$ are highly non-commutative. 
Thus a perturbation to one map may gives rise to some unexpected 
results. To overcome this problem, we first start from a dynamics 
which have a lot of abelian behaviors.
More precisely, we start $f$ and $g$ being translations: 
Imagine the situation where $f(x) = x+a$ and $g(x) = x-b$. 
This dynamics is easy to understand. 
An important quantity which governs this dynamics 
is the ratio of the translation $a/b$. If this value is irrational, 
every point has dense orbit and if it is rational, the dynamics has 
a good fundamental domain. 
The convenient situation for us is the rational case. 
Our goal is to construct hiding regions. 
For that, we just need to take a interval from the fundamental domain 
and take its images. Then it turns to be an invariant region. 

However, this is not enough for our purpose. 
We required that $f$ has contracting fixed point at $0$.
Thus we need to make the graph of $f$ come near the diagonal,
and this gives rise to another problem. To approach the diagonal, 
$f$ must have some hyperbolic (contracting), chaotic region, 
which destroys the abelian property and brings some problem. 
Our strategy to circumvent this problem is the ``divide and conquer."
We first start from the translation  $f(x) = x+a$ and $g(x) = x-a$
(where the ratio is equal to one). At this moment, we can calculate 
how large contracting behavior is required to reach the diagonal. 
Then we prove a proposition which says that
``if the hyperbolic behavior 
is divided into sufficiently small pieces, then we can construct 
a local hiding region by adding some buffer regions 
around the hyperbolic regions." 
See the argument in Section \ref{ss.conc}.

Unfortunately, at this moment we encounter another type of difficulty. 
The ``smallness" mentioned above involves the size of fundamental domain and 
it is determined by the denominator of the ratio of translation
(see the statement of Proposition \ref{p.exquant}). 
If the size of fundamental domain is too small, then we must make larger 
perturbation and that can make the resulted dynamics away from the identity map.
Here we have two contradicted demands: We want to divide the hyperbolic 
regions sufficiently small, keeping the denominators of ratio
which appear in the decomposition relatively small. 
Our first step to find such a convenient sequence using elementary number theory
(see Proposition \ref{p.goodrat}). 

Then the proof is reduced to the local problems 
(Proposition \ref{p.exrunway}, \ref{p.exconn} and \ref{p.exquant} ), 
which are the main theme after section 4. 

\subsection{Three local models}

In this section, we introduce three local IFSs called runway, connector and 
quantum leap. 

\begin{defin}
A \emph{runway} is a semigroup acting on $\mathbb{R}_{\geq 0}$
generated by $(f, g)$, where $f, g : \mathbb{R}_{\geq 0} \to 
\mathbb{R}_{\geq 0}$ satisfies the following:
\begin{enumerate}
\item $f(x) =(1-1/n)x$ in some neighborhood of $0$, where $n$ is 
a positive integer.
\item There exists $C >0$ such that $f(x) = x -1$ on $[C, +\infty)$. 
We call this interval \emph{right translation interval}.
\item $g(x) = x + d$, where $d$ is some positive integer.
\end{enumerate}
\end{defin}

\begin{defin}
Let $r, s \in \mathbb{Q}_{\geq 1}$.
An $(r, s)$-\emph{quantum leap} is a semigroup
acting on $\mathbb{R}$ and generated by
$(f, g)$ such that there exists $C >0$ 
satisfying the following:
\begin{enumerate}
\item $f(x) =x + r$ on $(C, +\infty)$.
We call this interval \emph{right translation interval}.
\item $f(x) = x + s$ on $(-\infty,  C]$.
We call this interval \emph{left translation interval}.
\item Outside $[-C, C]$, $g(x) = x-1$.
\end{enumerate}
In the case where $r=s$, we call $(f, g)$ \emph{$r$-connector}.
\end{defin}

Let $S^1 := \mathbb{R} / \mathbb{Z}$.
We define a map $\psi_{m, n}$ from  $S^1$ to
the target $\mathbb{R}_{\geq 0}$ or $\mathbb{R}$
by the formula $\psi_{m, n}(x) := (x+m)/n$, where $x \in [0, 1)$.
Let $L \subset S^1 \setminus \{ 0\}$ be a compact set. 
$K$ is said to be {\it periodic on the left translation area with
 shape $L$ and period $1/n$} 
if there exists $C$ such that for every 
$m+1 < C$, 
$K \cap [m/n, (m+1)/n) = \psi_{m, n}(L)$ holds.
We define the notion {\it periodic at the right translation area} similarly. 

Then, we prove the following three propositions.
\begin{prop} \label{p.exrunway}
For every natural number $\omega \geq 3$,
there exists a runway $(f, g)$ satisfying the following:
\begin{enumerate}
\item $\mu(f), \mu(g) \leq 1/\omega$.
\item $(f, g)$ has a hiding region $K \subset \mathbb{R}_{\geq 0}$ which is periodic 
on the right translation area with period $1$.
\end{enumerate}
\end{prop}

\begin{prop}\label{p.exconn}
Given $\omega \in \mathbb{N}_{>0}$, $r = p/q \in \mathbb{Q}_{\geq 1}$
and two regions $K_1, K_2 \subset S^1 \setminus \{ 0 \}$,
there exists an $r$-connector $(f, g)$
satisfying the following:
\begin{enumerate}
\item $\mu(f), \mu(g) < 1/\omega$.
\item $(f, g)$ has a hiding region $K$
which is periodic on the left translation with 
shape $K_1$, period $1/q$, and 
on the right translation area with 
shape $K_2$, period $1/q$.
\end{enumerate}
\end{prop}

To state the third one, we need some preparation. 
Recall that the {\it Farey series} $\mathfrak{F}_n$  
 is a finite ascending sequence of 
rational numbers whose irreducible representation 
is given by integers do not exceed $n$. 
See for example \cite{HW} for detail.
Then, the third one is the following:
\begin{prop}\label{p.exquant}
For every $\omega \geq 3$,
there exists $N (\omega) = N \in \mathbb{N}_{>0}$ such that 
for every $n \geq N $ the following holds: 
If  $r = h/k > r' = h'/k'$ 
are adjacent numbers in $\mathfrak{F}_n$
(we assume $h/k$, $h'/k'$ are irreducible representations ), 
then for every non negative 
integer $m$ there exists an $(r+m, r'+m)$-quantum leap 
$(f, g)$ satisfying the following:
\begin{enumerate}
\item $\mu(f), \mu(g) \leq 1/\omega$.
\item $(f, g)$ has a hiding region $K \subset \mathbb{R}$ which is periodic 
on the left translation area
with period $1/k$  and 
on the right translation area with period $1/k'$.  
\end{enumerate}  
\end{prop}

Compared to the previous two propositions, 
the intuitive meaning of Proposition \ref{p.exquant} is less clear. 
We will discuss it in the end of Section \ref{ss.farey}.
We do not give the proof of these three propositions here.
Proposition \ref{p.exrunway} will be proved in Section \ref{s.crun},
Proposition \ref{p.exconn} will be in Section \ref{s.con}, 
and Proposition \ref{p.exquant} in Section \ref{s.qua}.
In the rest of this section, we see how to prove Theorem assuming these propositions.

\subsection{Some properties of Farey sequences} \label{ss.farey}

In this subsection, we give a brief review of Farey sequence 
and explain why Farey sequence is useful for us. 
First, we cite some well-known properties of Farey sequences.
For the detail, see for example \cite{HW}.

\begin{lem} \label{l.farey}
Let $h/k > h'/k'$ be irreducible representations of 
two consecutive terms in $\mathfrak{F}_n$. Then,
\begin{enumerate}
\item $h/k -h'/k' = 1/(kk')$.
\item $ k+ k' \geq n+1$.
\end{enumerate}
\end{lem}

As a direct consequence of these properties, 
we have the following:
\begin{prop}\label{p.goodrat}
For every positive real number $\lambda$,
there exists $N(\lambda) = N \in \mathbb{N}_{>0}$
such that for every $n \geq N$ we have the following:
Let $r = h/k > r' = h'/k'$ 
be adjacent numbers in $\mathfrak{F}_n$
(we assume they are irreducible representations). 
Then we have the following inequality:
\[
 r -r' <  \frac{\lambda }{\min \{k, k'\}}.
\]
Furthermore, the following number 
\[
\frac{1}{ \min \{k, k'\}} \cdot  \frac{1}{(r -r')} 
\]
is an integer.
\end{prop}

\begin{proof}
Let $\lambda >0$ be given.
Then take $N$ satisfying $N > 2\lambda^{-1} -1$. 
We show that for all $n \geq N$, the sequence
$\mathfrak{F}_n$ satisfies the desired properties. 
Indeed, given two consecutive terms satisfying $h/k > h'/k'$,
by item 2 of Lemma \ref{l.farey} and choice of $n$, 
we have $\max \{ k, k' \} \geq (n+1)/2 > \lambda^{-1}$. Thus 
by item 1 of Lemma \ref{l.farey}, we have $h/k -h'/k' = 
1/(kk') < \lambda/ \min \{k, k'\}$.
Furthermore, by a direct calculation,
one can see $\min \{k, k'\}\cdot (r -r') = \max \{k, k'\}^{-1}$.
Thus we have that the reciprocal of that is an integer.
\end{proof}

Now, let us discuss the intuitive meaning of 
Proposition \ref{p.exquant}.
A very rough idea of Proposition \ref{p.exquant} is that 
``if two rational numbers $s > t$ are very close, then 
there exists $(s, t)$-quantum leap with non-trivial hiding regions.''
Indeed, this statement sounds quite plausible, since 
in the case $s = t$, it is very easy to find non-trivial attracting 
regions (take any non-trivial sub interval contained in the 
fundamental domain). If $s$ and $t$ are very near, then 
we may well expect that such a ``perturbed'' system may 
exhibit similar property. 

However, the situation is not such simple. The meaning 
of smallness depends on not only the 
absolute difference between $s$ and $t$, 
but also the denominator 
of $s$ and $t$ in irreducible representations
(remember that it governs the size of fundamental domains). 

In general, we can prove the following;
\begin{prop}\label{p.omitted}
For every $\omega \geq 3$, 
there exists $\lambda (\omega) = \lambda \in \mathbb{R}_{>0}$ such that 
the following holds:
For $r_1, r_2 \in \mathbb{Q}_{\geq 1}$, take their 
irreducible representations  $r_1 := p_1/ q_1$, $r_2 := p_1 /q_1$.
If they satisfy $ 0 < r_1 -r_2 < \lambda /\min \{q_1, q_2\}$, 
then there exists an $(r_1, r_2)$-quantum leap
with the following properties:
\begin{enumerate}
\item $\mu(f), \mu(g) \leq 1/\omega$.
\item $(f, g)$ has a hiding region $K \subset \mathbb{R}$ which is periodic 
on the left translation area
with period $1/q_1$  and 
on the right translation area with period $1/q_2$.  
\end{enumerate}  
\end{prop}
This proposition is more general since there is less 
restrictions on the choice of $r_i$. 
However, the proof of this proposition requires more
elaboration. Thus in this article we only furnish the proof of 
Proposition \ref{p.exquant}. 
On the other hand, the proof of Proposition \ref{p.omitted}
seems to have broader possibility of further application 
for the statement itself and the proof has more intuitive, 
geometric flavors. So the author is thinking of presenting 
it in another place.

Finally, we would like to point out the following.
By virtue of Farey series, we could prove 
Proposition \ref{p.goodrat} in a simple way. 
Meanwhile, it is worth mentioning 
that the proof has some relationship 
with Diophantine approximation theory.
For example, if we change the denominator in the right hand side of 
the inequality in \ref{p.goodrat}
to $(\min \{q_n, q_{n+1}\})^{2}$, then the corresponding statement is false
for sufficiently small $\lambda$.

\subsection{Conclusion}\label{ss.conc}
Using Proposition \ref{p.exrunway},  
\ref{p.exconn} and \ref{p.exquant},
let us prove the Theorem. 

\begin{proof}[Proof of Theorem]
To prove Theorem,  
we only need to construct the non-minimal PL semigroup actions 
with arbitrarily small $\mu$, see Proposition \ref{p.reduPL}. 
We fix $\omega \geq 3$. 
Our goal is to construct $(f, g) \in \mathcal{P}$
with hiding region such that $\mu(f), \mu(g) \leq 1/\omega$.

First, by Proposition \ref{p.exrunway} 
we pick up a runway $(f_r, g_r)$ with hiding region 
and $\mu(f_r), \mu(g_r) \leq 1/\omega$. 
We denote the shape of the hiding region
in the left translation region by $R_{-1}$ 
and the size of translation of $g_r$ in the left translation 
by $d_\omega$.
We also fix $ N = N (\omega)$ in Proposition \ref{p.exquant}
and fix $N_0 \geq N$. by using Proposition \ref{p.goodrat}, 
We take a sequence of decreasing rational numbers
$\{r_{n}\}_{n = 0, \ldots, m_0-1}$ from $d_\omega$ to $1$
as follows: First take $\mathfrak{F}_{N_0} \subset [0, 1]$.
Then consider the translation 
\[
\mathfrak{F}_{N_0} +m  := \{ r +m \mid r \in  \mathfrak{F}_{N_0}\}
\subset [m, m+1].
\]
Then, re-order $\{\mathfrak{F}_{N_0} +m\}_{m=1,\ldots, d_{\omega} -1}$ to the descending order and call it 
$\{r_n\}_{n = 0, \ldots, m_0 -1}$. Thus we have $r_0 = d_{\omega}$
and $r_{m_0 -1} = 1$.

Now, by applying Proposition \ref{p.exquant}, 
for each $n=0, \ldots, m_0-1$ 
we take $(r_{n}, r_{n+1})$-quantum leap
with hiding regions satisfying $\mu(f_n), \mu(g_n) \leq 1/\omega$.
For each $(f_n, g_n)$,  we have the shape 
of the hiding region in the left and right translation. 
We denote them by $L_n$ and $R_n$ respectively.
Then we apply Proposition \ref{p.exconn} to
pick up connectors $(\tilde{f}_n, \tilde{g}_n)$ 
between $R_{n-1}$ and $L_n$ for $n = 0, \ldots, m_0 -1$, 
satisfying $\mu(\tilde{f}_n), \mu(\tilde{g}_n) < 1/\omega$.
Finally, we take a connector 
$(\tilde{f}_{m_0}, \tilde{g}_{m_0})$ which connects 
$R_{m_0-1}$ with $S$ where $S$ is 
some non-empty region in $[0, 1]$ 
which is symmetric with respect to the point $\{1/2\}$.

We assemble them. More precisely, 
we construct new IFSs on $\mathbb{R}_{\geq 0}$ as follows.
We start the assembly of the runway $(f_r, g_r)$ and 
the connector $(\tilde{f}_0, \tilde{g}_0)$ between $R_{-1}$ and $L_0$.
From the right translation area of the runway $(f_r, g_r)$, we pick up 
an interval $P := [p, p+d_{\omega}]$  $(p \in \mathbb{Z})$ such that 
$f_r(P)$, $g_r(P)$ are contained in the right translation area. 
Similarly, we pick up an interval
$P' := [p', p'+d_{\omega}]$ $(p' \in \mathbb{Z})$ in the 
left translation area of the connector $(\tilde{f}_0, \tilde{g}_0)$ satisfying 
that the image of it $\tilde{f}_0, \tilde{g}_0$ are contained in 
the translation area. 
Then, we restrict $(f_r, g_r)$  to $[0, p+d_\omega]$, 
$(\tilde{f}_0, \tilde{g}_0)$ to $[p', +\infty]$ and
take a quotient space by identifying $P$ and $P'$ in a natural fashion.
On this quotient space (which is naturally homeomorphic to $\mathbb{R}_{\geq 0}$), 
we can naturally define a new dynamical 
systems which was the glued map of runway and the connector. 
Finally, we take the projection of the hiding regions to this quotient space 
from the old ones: This turns to be a region in the new space and 
we can check that this is a hiding region for new dynamics
(this is why we needed the ``margin"). Note that this construction does not 
change the value of $\mu$.

We continue this construction: Paste the one above with 
quantum leap $(f_0, g_0)$, then paste it with $(\tilde{f}_1, \tilde{g}_1)$,
and continue this process to obtain a dynamics $(F, G)$ with
hiding region $H$ which is $1$-periodic on the right translation area with 
its symmetric with respect to $\{ 1/2 \}$.
Remember that the size of the translation of $F$ is $1$ and 
of $G$ is $-1$. Furthermore, 
Now we pick up the ``mirror image" of $(F, G)$. 
Namely, we take a two PL-map on the left half line 
$\tilde{F}(x) := -F(-x) $ and
$\tilde{G}(x) := -G(-x) $. Then the mirror image $\tilde{H} := -H$
is the hiding region for $(\tilde{F}, \tilde{G})$.
Finally, by assembling $(F, G)$ and $(\tilde{F}, \tilde{G})$
we obtain an IFS on the interval with hiding region.
A priori, the length of the interval of this resulted dynamics is 
very large. So we rescale it to make it $1$. 
Note the rescaling does not affect 
the value of $\mu$. Thus we completed the proof. 
\end{proof}

\section{Auxiliary constructions}
In this section, we give some Auxiliary constructions
to prove Proposition \ref{p.exrunway},  
\ref{p.exconn} and \ref{p.exquant}

\subsection{Template}\label{s.temp}

In this subsection, we formulate 
a special kind of regions called templates. 
They have combinatorial properties 
convenient to establish the hiding properties under the action of
a contracting map and translation map. 
Such a simultaneous hiding property 
is very useful for the construction of runways and quantum leaps. 
We also furnish the proof of the existence of such regions. 

\begin{defin}
We put $f(x) := (1-1/n)x$ and $g(x) := x -1$.
Let $n$ be an integer equal to or greater than $3$. 
An $n$-\emph{template} is a disjoint union 
of finite number of
closed intervals $\mathcal{T} := \coprod T_i$ in $(0, n)$ 
($T_i$ are ordered in the increasing order)
 satisfying the following:
\begin{enumerate}
\item $T_m$ is the only interval contained in $(n-1, n)$.
\item $T_0$ is the only interval  contained in $(0, 1)$.
\item  For every $T_i$ ($i = 0,\ldots, m-1$), 
there exists $i'$ such that $T_i \subset f(T_{i'})$ holds.
\item For every $T_i$ ($i = 0,\ldots, m-1$), 
there exists $i'$ such that $T_i \subset g(T_{i'})$ holds.
\end{enumerate}
Furthermore, for each $i$  $(0 \leq i \leq n -1)$,
we put $\mathcal{T}_i := \mathcal{T} \cap [i, i+1]$.
\end{defin}

\begin{rem}
\begin{enumerate}
\item Suppose $\mathcal{T} := \coprod T_i $ is an $n$-template. 
Then for each $k$ ($k =0, \ldots, n-1$), 
there exists $j_k$ such that $T_0  \subset f^k(T_{j_k})$.
\item By definition, we have $\mathcal{T}_0 = T_0$
and $\mathcal{T}_m = T_m$.
\end{enumerate}
\end{rem}

\begin{prop}
For every $n \geq 3$, there exists an $n$-template.
\end{prop}

\begin{proof}
First, let $l$ be the smallest integer that satisfies $f^l(n) < 1$.
For a non-empty finite set of points  $X \subset [1, n]$,
We define 
\begin{align*}
C(X) &:= \cup_{x \in X} \{ f^i(x) \mid i >0, f^i(x) \in [f^l(n), n]  \}, \\
T(X) &:=  \cup_{x \in X} \{ g^i(x) \mid i >0, g^i(x) \in [f^l(n), n]  \}.
\end{align*}
If $X$ is an empty set, we put $C(X) = T(X) :=\emptyset$.
Note that these two sets are finite set contained in $[1, n]$.

For $C(X)$ and $T(X)$, we have the following estimate:
that $\max C(X) \leq \max X  -1/n$
and $\max T(X) =  \max X  - 1$ if $C(X)$, $T(X)$ are 
not empty sets. 
The first one comes from the 
fact that for every $x \in [1, n]$ we have $x -f(x)  = x/n \geq 1/n$,
and the second one is easy to see.

Now, we construct $E_k$ inductively as follows:
$E_0 :=\{ n\}$, $E'_m := E_m \cap [1, n]$ and 
$E_{m+1} := C(E'_m) \cup T(E'_m)$.
Since the maximum of $E_m$ decreases uniformly,
there exists $M$ such that 
$E'_M$ is non-empty and $E'_{M+1}$ is empty. 
Then take $\tilde{P} := \cup_{0 \leq m \leq M} E_{m}$,
$p_0 := \max (\tilde{P} \cap (0, 1))$ and 
$P := \{ p_0 \} \cup (\tilde{P} \cap [1, n])$. Note that $P$ is a finite set.
We claim that $P$ satisfies the following conditions:
For every $x \in P \setminus \{n\} $, there exists $q \in P$ such that
$f(q) =x$ or $g(q) = x$. It is clear for $x \in P \cap [1, n]$, since
$x$ belongs to $E_i$ ($i \geq1$) for some $i$ and the points which hit $x$
is not removed when we construct $P$.
If $x \in P \cap (0, 1)$, we can also check it since 
in this case we have $x = p_0$.

Align the elements of $P$ in the increasing order 
and denote the sequence by $\{p_i\}_{i = 0,\ldots, m}$.
Now let us take the sequence of open 
sets $ \{ P_{k} := (p_k, p_{k+1}) \mid  k=0, \ldots, m-1\}$. 
Note that, by construction, $P_m$ is the only interval in $(n-1, n)$
and $P_0$ is the only interval in $(0, 1)$.

We can check that these open sets satisfies the 
following condition: For  every $P_k$ ($k=0, \ldots, m-2$)
there exists $k' \in [0, m-1]$ such that $P_k \subset f(P_{k'})$.
Indeed, given $P_k = (p_k, p_{k+1})$, take  the largest $p_l$
in $P$ such that $f(p_l) \leq p_{k} $ holds 
(there exists at least one such $p_l$).
Then we show that $ f(p_{l+1}) \geq p_{k+1}$. 
Indeed, if not, we have $f(p_{l+1}) <  p_{k+1}$, in particular 
we have $f(p_{l+1}) \leq p_k$. 
But this contradicts to the choice of $p_l$.
Thus we have $P_k \subset f(P_{l})$.
By a similar argument, we can check that 
for every $P_k$ ($k=0, \ldots, m-2$)
there exists $k'$ such that $P_k \subset f(P_{k'})$.
Similarly, we can check that for  every $P_k$ ($k=0, \ldots, m-2$)
there exists $k'$ such that $P_k \subset g(P_{k'})$.

Thus these $\{P_k \}$ satisfies the hiding property both for $f$ and $g$.
To make them compact, we only need to shrink each $P_i$ a little bit. 
more precisely, first we shrink a little bit $P_0$ to obtain $T_0$.
Suppose we have constructed $T_k$ ($k = 0, \ldots, l$). Then
we construct $T_{k+1}$ as follows: $f(P_{k+1})$ and $g(P_{k+1})$
covers some of $T_{i}$. Then if we pick up $T_{k+1}$ sufficiently
close to $P_{k+1}$ so that $f(T_{k+1})$ and $g(T_{k+1})$ still cover them.
By continuing this process, we obtain $n$-template.
\end{proof}

\begin{rem}
The proof above is constructive, For example, Figure \ref{fig.aliint}
shows how $\{p_i\}$ in the proof align in $[0, 4]$ for $\omega =4$.
In general, for the right areas near $1$, there are many intervals.

\begin{figure}
\centering
\includegraphics{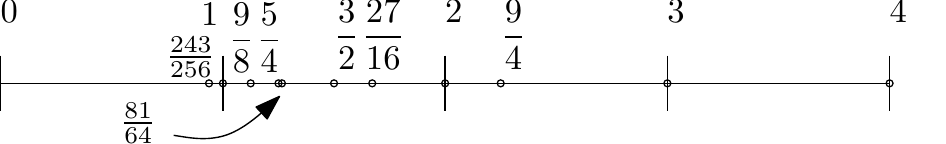}
 \caption{Alignment of $M_I$ for $\omega = 4$.}
\label{fig.aliint}
\end{figure}
\end{rem}

\subsection{Geometry deformation}

To construct the PL-maps with hiding region  whose 
slope is close to one, the basic idea 
is to modify the intervals gradually.
We give a precise definition of this notion. 

\begin{defin}
Suppose $K, L  \subset I$ are regions 
satisfying $K \subset L \subset \mathrm{int}(I)$ and $\varepsilon$
be a positive real number.
Then $K$ and $L$ are said to be {\it $\varepsilon$-equivalent}
if there exist two orientation preserving 
PL-homeomorphisms
$\varphi, \psi: I \to I$ such that the following holds:
\begin{itemize}
\item $\mu(\varphi)$, $\mu(\psi) < \varepsilon$.
\item $\varphi(K) \subset L$, $\psi(L) \subset K$.
\end{itemize}
\end{defin}

\begin{rem}
With the notion of $\varepsilon$-equivalence,
we can ``measure'' the closeness between 
two regions in $I$. Namely, we regard two regions are close if 
they are $\varepsilon$-equivalent for small $\varepsilon$. 
Note that this notion is essentially different from 
the Hausdorff distance. 
For example, take a connected interval $J \subset I$.
From $J$, remove a tiny interval centered at 
the middle point of $J$ and denote it by $J'$
(in other words, chop $J$ into two pieces). 
Then $J$ and $J'$ have very small Hausdorff distance
(as long as the removed intervals are small) 
but cannot be $\varepsilon$-equivalent for small $\varepsilon$.
\end{rem}

\begin{lem}\label{l.connect}
For every pair of regions $K, L \subset \mathrm{int}(I)$ 
and $\varepsilon >0$, there exists 
a sequence of regions $\{ \mathcal{M}_i\}$ ($i=0, \ldots, i_0$) satisfying the following:
$\mathcal{M}_0 =K$ and $\mathcal{M}_{i_0} = L$ and 
$\mathcal{M}_i$ and $\mathcal{M}_{i+1}$ are 
$\varepsilon$-equivalent for $i = 0, \ldots, i_0 -1$.
\end{lem}

\begin{proof}
We only give the rough idea of the proof:
If the number of the connected components of $K$ and $L$ are 
equal, then it is easy: We just need to expand or shrink $K$ to $L$.
It is not difficult to change the number of connected components: 
Given an interval, we can increase the number by adding a tiny interval
(remember that regions are always assumed to be non-empty).

More precisely, we can show the following;
suppose $K$ is a region with $m$-connected 
component. Then, for every $\varepsilon > 0$ and $n >0$,
it is $\varepsilon$-equivalent to $K'$ with  
$(m+n)$-connected components.
Indeed, take a connected component $K_1$ of $K$. 
then, we pick up a map $\varphi$ which expands $K_1$, 
$\mu(\varphi) < \varepsilon$ such that outside some neighbor 
hood of $K_1$ disjoint from other 
connected components,  $\varphi$ is an identity map. 
Then, take (possibly small) $n$ intervals 
$\{ L_i\}_{i =0,\ldots, n-1}$ from 
$\varphi(K_1) \setminus K_1$ and put 
$K' := K \coprod \left( \coprod L_i \right)$. Then 
then the identity map and $\varphi$ give the $\varepsilon$-equivalence.
\end{proof}

\begin{defin}\label{d.eseq}
We call the sequence $\{\mathcal{M}_i\}$ in Lemma \ref{l.connect}
\emph{an $\varepsilon$-equivalent sequence}
between $K$ and $L$ and $i_0$ is called \emph{length} of this sequence. 
Furthermore, by definition 
there exist sequences of maps $\{\varphi_i\}$,  $\{\psi_i\}$
$(i = 0,\ldots, i_0 -1)$ such that 
$\mathcal{M}_{i+1} \subset \varphi_i(\mathcal{M}_i) $ and
$\mathcal{M}_{i} \subset \psi_i(\mathcal{M}_{i+1})$
for every $i$. 
We call these maps \emph{connecting maps}.
\end{defin}

\begin{rem}\label{r.connect}
In Lemma \ref{l.connect}, 
suppose $K$ and $L$ are regions that consist of 
single connected components and they satisfy $K \subset L$. 
Then we can take above $\mathcal{M}_i$ 
so that $\mathcal{M}_i \subset \mathcal{M}_{i+1}$ holds for 
every $i = 0, \ldots, i_0 -1$.
\end{rem}

\subsection{Geometry eater}\label{s.geoe}
Template is a convenient geometric configuration of intervals 
to obtain the hiding property.
However, in general they have very complicated geometries 
that it is far away from periodic rotation, which is our final destination. 
In this part we prepare a perturbation which enables 
us to annihilate such complexity. 

\begin{defin}
Let $\mathcal{T} := \coprod T_i \subset [0, \omega]$ be 
an $\omega$-template. 
Remember that we put $\mathcal{T}_i := \mathcal{T} \cap [i, i+1]$.
To denote $\pi(\mathcal{T}_i)$ and $\pi(T_i)$, where
$\pi : [0, \omega] \to I = [0, 1)$ is the natural projection 
under the identification 
$x \sim y$ if $x-y \in \mathbb{Z}$,
we also use the same symbol $\mathcal{T}_i$ or $T_i$.
Then, an orientation preserving homeomorphism $h : I \to I$ 
is called an $\varepsilon$-{\it geometry eater} if the following holds:
\begin{itemize}
\item $\mu(h) < \varepsilon$.
\item There exists $n_0 >0$ such that for each $\mathcal{T}_i$
($i = 0, \ldots, \omega -1$), 
$h^{n_0}(\mathcal{T}_i)$ is $\varepsilon$-equivalent to 
$T_m$.
\end{itemize}
We call $n_0$ \emph{length} of the geometry eater.

\end{defin}

We have the following:
\begin{prop}
For each $\omega$-template $\mathcal{T} := \coprod T_i \subset [0, \omega]$ 
and $\varepsilon >0$ there exists an $\varepsilon$-geometry eater $h$.
\end{prop}

\begin{proof}
Again we avoid the formal proof, but just suggest rough 
idea of the proof. Let $h$ be a piecewise linear map which has 
\begin{itemize}
\item expanding fixed point $p$ in $T_0$, 
\item two contracting fixed point $q_1$, $q_2$ outside $T_m$,
but close to the endpoints of $T_m$, and
\item no fixed points between $p$ and $q_1$, $p$ and $q_2$. 
\end{itemize}
Then, for each $\mathcal{T}_i$, $h^n(\mathcal{T}_i)$ has the following feature:
there is one big interval, and rest of them are very small gathered close to $q_1$.
So, choosing $h$ appropriately, for some $n_0$
then we have
$h^{n_0}(\mathcal{T}_i)$ ($i =1, \ldots, \omega -2$)
 are all $\varepsilon$-equivalent to $T_m$. 
\end{proof}

\begin{rem}
Since $h$ is a homeomorphism, we have
$h(\mathcal{T}_i) \subset h(\mathcal{T}_{i+1})$ for every $i = 0,\ldots, \omega -2$.
\end{rem}

\subsection{Notations for defining maps}

In the rest of this article, we define many maps between one dimensional spaces.
For the effective description of them, we use 
following notations. 

Let $X, Y$ be one-dimensional connected manifolds
and $J \subset X$ be an interval.
Consider the set of maps $\mathcal{J} :=\{f : J \to \mathbb{R}  \}$.
We can introduce an additive group structure on $\mathcal{J}$
by pointwise operation. Then the set of constant maps 
$\mathcal{C}$ forms a subgroup in $\mathcal{J}$.  
We take the quotient group $\mathcal{J}^{\dagger} := \mathcal{J} / \mathcal{C}$.
Every map $f:J \to Y$  gives rise to an element 
$f^{\dagger} \in \mathcal{J}^{\dagger}$
(in the case $Y := S^1$ take the universal cover of 
$Y$ and take the lift of $f$). 
Note that, given an element in 
$g^{\dagger} \in \mathcal{J}^{\dagger}$ and the pair of points $(x, y) \in X \times Y$,
there is a unique map $G : J \to Y$,
which satisfies $G^{\dagger} = g^{\dagger}$ ($G^{\dagger}$ is the element in $\mathcal{J}^{\dagger}$
which $G$ naturally defines) and $G(x) = y$. 

Based on this fact, we use the following 
notation. Let $g: J \to Y$ be a map, $h^{\dagger} \in  \mathcal{J}^{\dagger}$
and $x \in J$. By {\it modify $g$ to $G$ over $J$ 
into $h$ keeping $G(x) = g(x)$}
we mean defining $G$ in a unique way such that 
$G^{\dagger} = h^{\dagger}$ and $g(x) =y$ on $J$
and $G = g$ outside $J$ .
Furthermore, if 
$h^{\dagger} \in \mathcal{J}^{\dagger}$ satisfies
$\mathrm{length}(g(J)) = \mathrm{length}(h^{\dagger}(J))$,
then the sentence \emph{modify $g$ to $G$ over $J$ into $h^{\dagger}$}
means to modify 
$g$ to $G$  in a unique way so that $(G|_{J})^{\dagger}= h^{\dagger}$ and
$G(J) = g(J)$ holds. We denote this modification just by $G[J] = h^{\dagger}$ when what $g$ means is clear from the context.

In the next section, we keep using the symbol $\dagger$ to distinguish 
if an element is in $\mathcal{J}$ or $\mathcal{J}^{\dagger}$, but 
from Section \ref{s.con} we omit the use of 
it since it will not bring much trouble.

\section{Construction of runway}\label{s.crun}

In this section, we give the proof of Proposition \ref{p.exrunway}.

\begin{proof}[Proof of Proposition \ref{p.exrunway}]
First, we fix $\omega \geq 3$. Then we pick up an $\omega$-template 
$\mathcal{T} = \coprod T_i \subset [0, \omega]$. We denote by $\pi$
the natural projection 
$\pi : \mathbb{R}_{\geq 0} \to I$ which identifies 
$x \sim y$ if $x-y \in \mathbb{Z}$.
Take $T_0,T_m \subset I$.
By applying Lemma \ref{l.connect}, we take a $1/\omega$-sequence $ \{M_i\}$ 
between $T_0$ and $T_m$ with length $d$. By Remark \ref{r.connect},
we can assume $T_0 \subset M_{i} \subset M_{i+1} \subset T_m$
for every $i= 0,\ldots, i_0-1$. Let $(\phi_{i})$ denote the maps which 
satisfies $M_{i+1} \subset \phi_i(M_i)$ and $\mu(\phi_i) < 1/\omega$.

Now we construct $F$ and $G$ as follows. First, 
we construct an auxiliary map $f_1$ as follows.
\[
f_1(x) :=
\begin{cases}
(1-1/\omega)x & (0 < x < \omega) \\
x -1 & (x \geq \omega)
\end{cases}
\]
We define $G(x) := x + d$, where $d := \omega + i_0$ 
and $I(i) := [i, i+1] \subset \mathbb{R}_{\geq 0}$. 
Note that $f_1(I(i)) = I(i-1)$ for $i \geq \omega$
and $G(I(i)) = I(i+d)$.

Then, we perturb $f_1$ to $F$ as follows:
\begin{itemize}
\item $F[I( \omega -1 + ( i_0 -i))] = (\phi_{i})^{\dagger}$ for $0 \leq i \leq i_0 -1$, 
\item Outside above $I(i)$, put $F := f_1$.
\end{itemize}

We can easily check that $\mu(F), \mu(G) \leq 1/\omega$.

Let us define the hiding regions $K$. 
We take $K \subset \mathbb{R}_{\geq 0}$ 
so that the following holds: 
\begin{itemize}
\item $\pi(I(i) \cap K) = \pi(\mathcal{T}_i)$ for $i \geq d$
\, (for the definition of $\mathcal{T}_i$, see Section \ref{s.temp}).

\item $\pi(I(\omega -1 + (i_0 -i)) \cap K) = M_{ i}$ for $i=0, \ldots, i_0 -1$.
\item $\pi( I(i) \cap K) = M_0 = T_0$ for $i \geq \omega +i_0$.
\end{itemize}

We show that $K$ is a hiding region for $(F, G)$.
We first check the hiding property under $F$. 
Since $F(\mathbb{R}_{\geq 0}) = \mathbb{R}_{\geq0}$, we need to 
check that every connected component of $K$ is covered by 
the image of $K$. To check it, 
we use the following convenient criterion: Let $\{ I_i\}$ be a family of 
intervals such that $\cup I_i \subset \mathbb{R}_{\geq 0}$. 
Then, if for each $I_i$ we can check $F(I_i) \cap K \subset F(I_i \cap K)$,
then we have $K \subset F(K)$. 

Let us start checking that condition.
First, we can check the it on $[0, \omega]$
by the definition of template. 
Let us consider the case $I( \omega -1 + ( i_0 -i)))$ where 
$ i= 0, \ldots, i_0 -1 $.
In this case, $F[I(\omega -1 + ( i_0 -i))] = (\phi_{i})^{\dagger}$,  
$\pi(K\cap I(\omega -1 + ( i_0 -i))) = M_i$, and
$\pi(K\cap F(I(\omega -1 + ( i_0 -i)))) = M_{i+1}$.
So we have the hiding property. 
The case $ i \geq i_0 + \omega $ is clear.

The check for $G$ can be done similarly. 
First, we have $G(\mathbb{R}_{\geq 0}) = [d, +\infty)$.
For each $I(i)$ where $i \geq d$, $\pi(K\cap I(i)) = T_0$.
Since $\pi(K\cap I(i))$ are  (i) $\mathcal{T}_k$ for some $k$ or 
(ii) $M_k$ for some $k$ or (iii) $T_0$, for each cases, we can check the 
hiding property.
\end{proof}

\section{Construction of connector}\label{s.con}
The aim of this section is to prove Proposition \ref{p.exconn}. 
In short, our strategy is to give a perturbation on each $I(i)$
paying attention to the combinatorial information of 
translations. 

\begin{proof}[Proof of Proposition \ref{p.exconn}]

First, we prepare some notations.
Let $r =p/q \in \mathbb{Q}_{\geq 1}$ 
which we assume irreducible.
Note that since  $p/q \geq 1$, we have $p \geq q$.
We put $I(i) :=[i/q, (i+1)/q]$, and introduce following grouping:
\begin{itemize}
\item $H\{k\}    := \{ i \in \mathbb{Z} \mid k(p+q) \leq i < (k+1)(p+q)  \}$.
\item $H_l\{k\}  := \{ i \in \mathbb{Z} \mid k(p+q)  \leq i < k(p+q) + p \}$.
\item $H_r\{k\} := \{ i \in \mathbb{Z} \mid k(p+q) + p \leq i < (k+1)(p+q) \}$.
\end{itemize}

We put $f_1 (x) := x -1$ and $g_1(x)  :=x +p/q$. 
Then we have $f_1(I(i)) = I(i - q)$ and $g_1(I(i)) = I(i +p)$.
Note that the following:
\begin{itemize}
\item If $i \in H_l\{k\}$ then
$i-q \in H_l\{k\} \cup H_r\{k -1\} $.
\item If $i \in H_l\{k\}$ then $i +p \in H_r\{k\} \cup H_l\{k+1\}$.
\item If $i \in H_r\{k\}$ then 
$i-q \in H_l\{k\} $.
\item If $i \in H_r\{k\}$ then 
$i +p \in  H_r\{k \} \cup H_l\{k+1\}$.
\end{itemize}

Let $K_1$, $K_2 \subset S^1 \setminus \{0 \}$ be regions.
Fix $\omega >0$.
Take a natural map $\pi : S^1 \setminus \{0 \} \to [0, 1]$ and
consider $\pi(K_1)$, $\pi(K_2)$. Then by applying Lemma \ref{l.connect},
we can take the $\varepsilon$-equivalent sequence between 
$\pi(K_1)$ and $\pi(K_2)$. Take the natural lift of them to $S^1$.
We denote the sequence of maps that gives the equivalence
by $\{U_i\}$ and $\{V_i\}$ ($i = 0, \ldots, i_0 -1$).

Then, we renormalize them toto the circle 
$S^1 := \mathbb{R}/(1/q)\mathbb{Z}$ (see Remark \ref{r.renom}).
We denote the renormalized maps and regions by the same symbols. 
Let us start the perturbation.
We perturb $f_1$ intoto $F$ as follows:
\begin{itemize}
\item  On $I(i)$ where $i < 0$ or $i \geq i_0(p+q)$,  $F=f_1$. 
\item  On $I(i)$ where $i \in H_l\{k\}$ ($k = 0, \ldots, i_0-1$), 
$F[I(i)] = \mathrm{id}$.
\item  On $I(i)$ where $i \in H_r\{k\}$ ($k = 0, \ldots, i_0-1$),  
$F[I(i)] = V_k$.
\end{itemize}
The perturbation of $g_1$ into $G$ is given as follows
\begin{itemize}
\item  On $I(i)$ where $i < 0$ or $i \geq i_0(p+q)$,  $G=g_1$. 
\item  On $I(i)$ where $i \in H_l\{k\}$ ($k = 0, \ldots, i_0-1$), $G[I(i)] =  U_k$.
\item  On $I(i)$ where $i \in H_r\{k\}$ ($k = 0, \ldots, i_0-1$), 
$G[I(i)] = \mathrm{id}$.
\end{itemize}
By construction, we can easily see that $\mu(F), \mu(G) < 1/\omega$.

By $\pi$ we denote the natural projection 
$\pi : \mathbb{R} \to S^1 = \mathbb{R}/(1/q)\mathbb{Z}$.
This notation conflicts to that in Section \ref{s.crun},
but it would not bring much confusion.
Now we define the region $K \subset \mathbb{R}$ as follows: 
\begin{itemize}
\item On $I(i)$ where $i <  p$,  $\pi (K \cap I(i)) =M_0 = K_0$.
\item On $I(i)$ where $i \in H_r\{k\} \cup H_l\{k+1\}$ \,  $(k =0,\ldots, i_0-1)$,
we put $\pi(K \cap I(i)) = M_k$.
\item On $I(i)$ where $i \geq (i_0-1)(p+q) + p$, 
$\pi(K \cap I(i)) =  M_{i_0} =K_1$.
\end{itemize}

Let us see that this $K$ is a hiding region for $(F, G)$. 
In the following, we check it for $F$, that is, 
we check $(F(I(i)) \cap M) \subset  F(I(i) \cap M)$ holds for all $i$.
There are four cases to check:
\begin{itemize}
\item For $i < p$. In this case, 
we have $\pi (F(I(i)) \cap K) = M_0$. 
since $i -q< p$,
$\pi (F(I(i)) \cap K) = M_0$. 
Furthermore, we have $F[I(i)] = \mathrm{id}$. Thus we have the hiding property.

\item For $i  \in H_l\{ k\}$ ($k =  1, \ldots, i_0 -1$). 
In this case, 
we have $\pi (F(I(i)) \cap K) = M_k$. 
Since $i -q \in H_l\{k\} \cup H_r\{k -1\}$, we have
$\pi (F(I(i)) \cap K) = M_{k}$. 
Furthermore, we have $F[I(i)] = \mathrm{id}$. Thus have the hiding property.

\item For $i \in H_r\{ k\}$ ($k =  0, \ldots, i_0 -1$).
In this case, 
we have $\pi (F(I(i)) \cap K) = M_{k+1}$. 
Since $i -q \in H_l\{k\}$, we have
$\pi (F(I(i)) \cap K) = M_{k}$. 
Furthermore, we have $F[I(i)] = V_k$. Thus we have the hiding property.

\item For $i \geq i_0 (p+q)$. 
In this case, 
we have $\pi (F(I(i)) \cap K) = M_{i_0}$. 
Since $i -q \geq (i_0-1)(p+q) + p$, 
we have $\pi (F(I(i)) \cap K) = M_{i_0}$. 
Furthermore, we have $F[I(i)] = \mathrm{id}$. Thus we have the hiding property.
\end{itemize}
The check for $G$ can be done similarly, so we omit the proof of it. 
\end{proof}

\begin{rem}\label{r.renom}
Let us clarify the meaning of the renormalization. 
\begin{enumerate}
\item Let $X_1$, $X_2$ be intervals. 
Let $f: X_1 \to X_1$ be a map and $Y$ be a region in $X$.
Then, the \emph{renormalization of $f$ to $X_2$} is the map
$\tilde{f} :X_2 \to X_2$ given 
as follows: Take the (unique) affine orientation preserving homeomorphism 
$\phi :X_1 \to X_2$. Then $\tilde{f} := \phi \circ f \circ \phi^{-1}$.
The renormalization of $Y$ to $X_2$ is $\phi (Y)$.

\item Let $X_1 := \mathbb{R}/ a_1\mathbb{Z}$, 
$X_2 := \mathbb{R}/ a_2\mathbb{Z}$ (where $a_1, a_2$ denote some 
positive real numbers) be circles. 
Let $f: X_1 \to X_1$ be a map and $Y$ be a region in $X$.
Then, the \emph{renormalization of $f$ to $X_2$} is the map
$\tilde{f} :X_2 \to X_2$ given 
as follows: Take the affine orientation preserving homeomorphism 
$\phi : X_1 \to X_2$ defined by $\phi (x) := a_2/a_1x$.
Then $\tilde{f} := \phi \circ f \circ \phi^{-1}$.
The renormalization of $Y$ to $X_2$ is $\phi (Y)$.
\end{enumerate}

\end{rem}

\section{Construction of quantum leap (I)}\label{s.qua}

In this section, we construct quantum leaps assuming ``local constructions."
Note that {\it in this section and the next section, 
we consider the case $q_1 \leq q_2$}.
The proof of the case $q_1 > q_2$ can be done similarly, but 
requires small modifications. So we discuss it later (see Section \ref{s.anoth}).

\subsection{How should $N$ be?}\label{howlambda}
Proposition \ref{p.exquant}
involves a constant $N$ which only 
depends on the choice of $\omega$,
independent of $p_i/q_i$. 
As we have seen in Section \ref{ss.conc},
this independence is the heart of the proof of the Theorem.
In this subsection, we explain for given $\omega$ how should 
$N$ be fixed. 

First, we prepare a lemma.
\begin{lem}\label{l.decsmall}
Given $\varepsilon >0$, there exists $\varepsilon_1 >0$ 
such that the following holds: If $f, g$ satisfies 
$\mu (f), \mu (g)  < \varepsilon_1$, 
then  $\mu ( f \circ g) < \varepsilon$.
\end{lem}

We omit the proof of Lemma \ref{l.decsmall}. 
By applying this lemma, we fix $\Omega >0$
such that if $f, g$ satisfy
$\mu(f ), \mu(g) < \Omega$  
then  $\mu (f \circ g) < 1/\omega$ holds.

First, we take $\omega$-template. 
Then take the geometry eater 
$h$ with $\mu(h) < \Omega$ with length $\xi$
(see Section \ref{s.geoe}). This $\xi$ is the second constant.
Now we take $\lambda >0$ such that the following holds:
\[
\lambda < 1/  (\omega + \xi +1).
\]
Then finally take $N$ by applying Proposition \ref{p.goodrat} for 
this $\lambda$.

The important point of this estimate is that under 
this choice of $N$, every consecutive 
$r = h/k > h'/k' =  r'$ in $\mathfrak{F}_{N}$ satisfies
the following; 
\[
\tau := \frac{1}{ \min \{k, k'\}} \cdot \frac{1}{ (r - r')} \geq  \omega + \xi +2,
\]
and the left hand side is an integer.
This estimate is the heart of the constructions of Section \ref{s.qua2}.

\subsection{Orbit circle perturbation: Definition}

For the construction of the connector and the quantum leap, 
it is convenient to reduce our problem to a circle. 
The {\it orbit circle} is the ideal space 
which enables us to describe the perturbation in simpler fashion.

\begin{defin}
Let $\alpha \in S^1$.
A pair $(u, M)$, where $u :S^1 \to S^1$ is an 
orientation preserving PL-homeomorphism 
and $M$ is a region in $S^1$, is called an 
\emph{$\alpha$-hiding pair} if the following holds:
\begin{itemize}
\item $u(0) =-\alpha$
(on $S^1$ we endow the natural abelian group structure).
\item $M \subset u(M)$.
\end{itemize} 
Furthermore, we say that
$(u, M)$ satisfies \emph{pasting property}
is the following holds:
\begin{itemize}
\item $u(\alpha) =0$.
\item $0, \alpha \not\in M $.
\end{itemize} 
\end{defin}

We extend the concept of ``modifying $\alpha$-pair gradually"
to two hiding pairs.
\begin{defin}
Let $(u_0, M_0)$ and $(u_1, M_1)$ be $\alpha$-hiding pairs. 
We say that they are \emph{$\varepsilon$-equivalent} 
if there exists orientation preserving homeomorphism $U$, $V$ such that 
the following holds:
\begin{itemize}
\item $\mu(U), \mu (V) < \varepsilon$. 
\item $U(0) = -\alpha$ and $M_1 \subset U(M_0)$.
\item $V(0) = 0$ and $M_0 \subset V(M_1)$.
\end{itemize}
We say that $U$ satisfies \emph{pasting property} if 
$U(\alpha) =0$.

We define the $\varepsilon$-equivalent sequence of $\alpha$-hiding pairs, 
its length in a similar fashion as is in Definition \ref{d.eseq}.
\end{defin}

\begin{rem}
To construct a quantum leap, 
one important step is construct a sequence 
of $\varepsilon$-equivalent $\alpha$-hiding pairs which connects two 
specific pairs. Section 9 is devoted to solve this problem.
One may expect a decomposition lemma similar to Lemma \ref{l.connect}
holds for given two $\alpha$-hiding pairs. 
However, proving such a lemma 
in a general setting seems to require deep arguments. 
At the moment of writing, the author does not know if it is possible or not.
Thus in this article we do not pursue this problem.
\end{rem}

\begin{rem}\label{r.pasting}
Pasting properties is used as follows: 
if $(u, M)$ and $U$ above satisfies the pasting properties, 
then it implies 
that $u|_{[ 0, \alpha]}$, $U|_{[0, \alpha]}$ are both homeomorphism 
 from $[0,\alpha]$ to $[-\alpha, 0]$. 
Thus the following map $(U)^{\ast}$ is a well-defined homeomorphism on $S^1$:
\[
(U)^{\ast}(x) =
\begin{cases}
u(x)          & x \in [0, \alpha] \subset S^1, \\
U(x)          &  \mbox{\normalfont otherwise.}
\end{cases}
\]
\end{rem}

\subsection{Initial pair: Axiomatic description}

In this section, we give a definition of some maps which is used 
for the proof of Proposition \ref{p.exquant}.
In the following, we only give the properties 
of these maps and postpone the proof of the existence of such maps. 
The existence of them will be provided in Section \ref{s.qua2}. 

\begin{prop}\label{p.initiator}
Let $S^1 = \mathbb{R} /(1/q_1)\mathbb{Z}$. 
There exists an $R$-hiding pair $(u_I, M_I)$ with pasting property 
satisfying the following:
\begin{itemize}
\item $\mu(u_{I}) <1/\omega$.
\item $u_{I}(M_I) \subset M_I$.
\item $u_I|_{[0, \omega R]}(x) = x -R$.
\item Define $\nu : [0, 1/q_1) \to S_1$ as follows:
$\nu|_{[0, \omega R]} (x) = (1-1/\omega)x$ and 
$\nu|_{[R, 1/q_1)} = u_I$ (note that this is a 
homeomorphism on its image).
Then we have $\big ( M_I \cap [0, 1/q_1 -R] \big)\subset 
\nu (M_I) \cap [0, 1/q_1 -R] $.
\end{itemize}
we call such $(u_I, M_I)$ an \emph{$\omega$-initial pair}.
\end{prop}

This dynamics will be used to define IFSs 
which has the compatibility for both translation behavior and 
contracting behavior. 

\begin{defin}
An $\alpha$-hiding pair $(u, M)$ is called 
an \emph{$\alpha$-translation pair} if $u$ 
is a translation $u(x) = x -\alpha $.
\end{defin}

The construction of the quantum leap is crystallized in the 
following proposition.

\begin{prop}\label{p.elysian}
Let $(u_I, M)$ be an $\omega$-initial pair which is obtained by 
as a consequence of Proposition \ref{p.initiator}.
Then there exists a sequence of $1/\omega$-equivalent $R$-hiding 
pairs $\{ (u_i, M_i) \}_{i = 0,\ldots, i_0}$ satisfying 
$\mu (u_i) < 1/\omega$ for each $i$
which connects the $\omega$-initial pair $(u_I, M_I)$ 
and an $R$-translation pair. 
furthermore, we can take the connecting maps $(U_i)$
with pasting property.
\end{prop}

\subsection{Realization of orbit circle maps}\label{c.qua1}
Now assuming Proposition \ref{p.initiator} and \ref{p.elysian},
we prove the existence of the quantum leap.

We prepare some notations. 
Let $ r = r_1= p_1/q_1$, $r' = r_2 =p_2/q_2$ satisfying 
$0 < R :=r_1 - r_2 < \lambda/q_1$ be given
(remember that  we consider the case $q_1 \leq q_2$).

Then, we put as follows:
\begin{itemize}
\item $H\{i\} := \{ k \in \mathbb{Z} \mid (p_1 +q_1)i \leq k < (p_1 +q_1)(i+1) \}$.
\item $H^l\{i\} 
:= \{ k \in \mathbb{Z} \mid (p_1 +q_1)i \leq k < (p_1 +q_1)i +p_1 \}$.
\item $H^r\{i\} := \{ k \in \mathbb{Z} \mid (p_1 +q_1)i +p_1 \leq k < (p_1 +q_1)(i+1)\}$.
\end{itemize}

For each $i$, put $I(k, \ast) :=  [k(p_1+q_1)/q_1, k(p_1+q_1)/q_1 +R ] \subset 
I( k(p_1+q_1) )$ and call them \emph{sliding interval}.

Now, let us start the proof of Proposition \ref{p.exquant}.

\begin{proof}[Proof of Proposition \ref{p.exquant}]
First, we define $g_1$ as follows:
\[
g_1(x) =
\begin{cases}
x + r_1          & (x \leq 0) \\
(1-1/\omega)x +r_1  & (0 \leq x \leq  \omega R)   \\
x + r_2        & (x \geq \omega R)
\end{cases}.
\]
Note that $x+r_2 = x+ r_1 -R$.
We put $f_1(x) :=x -1$.
These are orientation preserving piecewise-linear homeomorphisms
on $\mathbb{R}$. 

Now by applying Proposition \ref{p.initiator},
we fix an $1/\omega$-initial pair $(u_I, M_I)$.
Furthermore, we apply 
Proposition \ref{p.elysian} to obtain a sequence
of $R$-hiding pair $(u_i, M_i)$ of length $i_0$ such that 
$(u_0, M_0) = (u_I, M_I)$ and  $(u_{i_0}, M_{i_0})$ is $R$-translation pair.
We denote the sequence of maps which gives the 
$1/\omega$-equivalence between $(u_i, M_i)$ and
$(u_{i+1}, M_{i+1})$ by $(U_i)$ and $(V_i)$  ($i = 0,\ldots, i_0 -1$).
Then, we define $U_i^{\ast}$ as follows (see Remark \ref{r.pasting}):
$(U_i)^{\ast}(x) = u_i(x)$  if $x \in [0, R] \subset S^1$ and
$(U_i)^{\ast}(x) = U_i(x)$ otherwise.

Then let us describe the modifications. 
In the following, $\pi$ denotes the natural projection from $\mathbb{R}$ to 
$S^1 := \mathbb{R} / (1/q_1)\mathbb{Z}$.

From $f_1$ to $F$:
\begin{itemize}
\item  On $I(i)$ where $i < 0$ or $i \geq (i_0 +1)(p_1+q_1)$,  $F=f_1$. 
\item  On $I(i)$ where $i \in H^l\{k+1\}$ ($k = 0, \ldots, i_0-1$), 
$F[I(i)] = \mathrm{id}$.
\item  On $I(i)$ where $i \in H^r\{k+1\}$ ($k = 0, \ldots, i_0-1$),  
$F[I(i)] = V_k$.
\end{itemize}

From $g_1$ to $G$:
\begin{itemize}
\item On $I(i)$ where $i < 0$ or $i \geq (i_0+1)(p_1+q_1)$, 
$G = g_1$.
\item On $I(0)$, $G[I(0)] = \nu$ (see the definition of initial pair).
\item On $I(i)$ where $i \in H\{0\} \setminus \{0 \}$,
$G[I(i)] = u_0$.
\item  On $I(i)$ where $i \in H^l\{k+1 \} \setminus \{ (k+1) (p_1 +q_1) \}$ 
$(0 \leq k < i_0 )$, $G[I(i)] = U_k$.
\item On $I((k+1)(p_1+q_1) )$ \, $(0 \leq k < i_0 )$, 
$G[I((k+1)(p_1+q_1))] = (U_k)^{\ast}$.
\item  On $I(i)$ where $i \in H^r\{k+1 \}$ ($0 \leq k < i_0 $), 
$G[I(i)] = u_{k+1}$.
\end{itemize}
By construction, we can check that $\mu(F), \mu(G) \leq 1/\omega$ easily.
We take the region $M$ as follows:
\begin{itemize}
\item On $I(i)$ where $i < p_1+q_1$, we define 
$M$ so that $\pi (M \cap I(i)) = M_0$ holds.
\item On $I(i)$ where $i \geq (i_0 +1)(p_1+q_1)$, we define $M$ so that 
$\pi(M \cap I(i)) = M_{i_0}$ holds.
\item On $I(i)$ where $i \in H^l\{ k+1\}$ ($ 0 \leq  k < i_0$), we define $M$ so that 
$\pi (M \cap I(i)) = M_k$ holds.
\item On $I(i)$ where $i \in H^r\{ k+1\}$ ($ 0 \leq  k < i_0$), we define $M$ so that 
$\pi (M \cap I(i)) = M_{k+1}$ holds.
\end{itemize}

Let us check the hiding property of $(F, G)$ for $M$. 
We need to check $M \subset F(M)$ and $M \subset G(M)$.
Checking the condition for $F$ is almost same as the proof of 
Proposition \ref{p.exconn}. So we omit that case and 
concentrate on the hiding property of $G$. 

As was in the proof of 
Proposition \ref{p.exconn}, we only need to check the hiding property 
on each interval $I(i)$. Let us check it. 

\begin{itemize}
\item If $i <0$, $G(I(i)) = I(i + p_1)$, $G[I(i)] = \mathrm{id}$ and 
$\pi (I(i) \cap M) =  \pi (I(i+p_1) \cap M) = M_0$. So we have the 
hiding property.
\item On $I(0)$. 
\begin{itemize}
\item On $[0, \omega R]$, 
$G([0, \omega R]) = [p_1/q_1, p_1/q_1 + (\omega-1)R ]$.
We also have $G\big[ [0, \omega R] \big] = (1-1/\omega)x$ and
$\pi ([0, \omega R]  \cap M) =  \pi ( G([0, \omega] \cap M) )
= M_I$. Thus by the definition of the initial pair, we have the hiding property.
\item On $I_0 :=I(0) \setminus [0, \omega R]$.
We know $G(I_0) =  [p_1/q_1 + \omega R,  (p_1 +1)/q_1 ] -R \subset I(p_1)$,
$G\big[ I_0 \big] = u_0|_{I_0}$ and 
$\pi ( I_0 \cap M) =  \pi ( G(I_0) \cap M)  =M_0$.
Thus we have the hiding property.
\end{itemize}
\item On $I(i)$ where $i \in H \{0\} \setminus \{0\}$. 
$G(I(i)) = I(i+p_1) - R$, $G[ I(i) ] = u_0$ and 
$\pi ( I(i) \cap M) = \pi ( G(I(i))  \cap M) = M_0$.
Thus we have the hiding property.
\item On $I((k+1)(p_1+q_1))$ where $0 \leq k \leq i_0 $.
Note that this region contains the sliding interval $I(k+1, \ast)$.
\begin{itemize}
\item On $I(k+1, \ast)$. We know $G(I(k+1, \ast)) \subset I(j)$
where $j \in H^l\{ k+1\}$.
$G[I(k+1, \ast)] = u_k|_{[0, R]}$,
$\pi (I(k+1, \ast) \cap M) = M_k \cap [0, R]$ and
$\pi (G(I(k+1, \ast)) \cap M) = M_k \cap [-R, 0]$.
Thus we have the hiding property.
\item On $I_{k+1} := I((k+1)(p_1+q_1)) \setminus I(k+1, \ast)$.
We have $G(I_{k+1}) \subset I(j)$ where $j \in H^r\{ k+1\}$,
$G[I_{k+1}] = U_k|_{[R, 1/q_1]}$,
$\pi (I_{k+1} \cap M) = M_k \cap [R, 1/q_1]$ and
$\pi (G(I_{k+1}) \cap M) = M_{k+1} \cap [0, 1/q_1-R]$.
Thus we have the hiding property.
\end{itemize}
\item On $I(i)$ for $i \in H^l\{k+1\} \setminus \{(k+1)(p_1+q_1) \}$ where 
$0 \leq k < i_0 $.
In this case, $G(I(i)) = I(i+p_1)-R$, $G[I(i)] = U_k$, 
$\pi (I(i) \cap M) = M_k $ and
$\pi (G(I(i)) \cap M) = M_{k+1}$. Thus we have the hiding property.

\item On $I(i)$ for $i \in H^r\{k+1\} $ where $0 \leq i < i_0 $.
In this case, $G(I(i)) = I(i+p_1)-R$, $G[I(i)] = u_{k+1}$, 
$\pi (I(i) \cap M) = \pi (G(I(i)) \cap M) = M_{k+1}$. 
Thus we have the hiding property.

\item Finally, If $i \geq (i_0+1)(p_1+q_1)$, then 
$G(I(i)) = I(i + p_1)-R$, $G[I(i)] = \mathrm{id}$ and 
$\pi (I(i) \cap M) =  \pi (I(i+p_1) \cap M) =M_{i_0}$.
Since $(u_{i_0}, M_{i_0})$ is $R$-periodic, 
we have the hiding property. 
\end{itemize}
Thus the construction is completed.
\end{proof}

\section{Construction of quantum leap (II)} \label{s.qua2}
Now we start the construction on the orbit circle. 
We need to do two things: Proof of the existence of 
initial pair (Proposition \ref{p.initiator}) and 
construction of the the $1/\omega$-equivalent sequence 
from the $\omega$-initial pair to the 
$R$-translation pair (Proposition \ref{p.elysian}).

\subsection{Notations}\label{s.notrot}
We prepare some notations on the intervals in 
$S^1 := \mathbb{R} / (1/q_1) \mathbb{Z}$. 

We put $I(i) : = [iR, (i+1)R]$
and $J(i) : = [ 1/q_1 -(i+1)R, 1/q_1 - iR]$. 
$\{I(i) \}$ divides $S^1$ into $\tau = 1/(q_1R)$ intervals 
and $I(i) = J(\tau -i-1)$ for $0 \leq i < \tau$.
We also have $I(0) -R = J(0)$.
In this section \ref{s.qua2},
we denote by $\pi$ the projection $\pi : [0, \tau R] \to [0, R)$
under the identification $x \sim y$ if $x-y \in R\mathbb{Z}$.

Finally, we fix some notations on the geometry eater. 
Remember that in the previous section we fixed a 
geometry eater $f$ for the $\omega$-template.
By definition, for every $i =0, \ldots, \omega -2$, 
$h^{\xi}(\mathcal{T}_i)$ are $\Omega$-equivalent 
to $\mathcal{T}_{\omega -1}$. We denote 
the maps which give this equivalence by $\alpha_i$
and $\beta_i$, that is, they satisfy $\mu(\alpha_i)$
$\mu(\beta_i) < \Omega$,
$\mathcal{T}_{\omega -1} \subset \alpha_i(h^{\xi}(\mathcal{T}_i))$
and 
$h^{\xi}(\mathcal{T}_i) \subset \beta_i(\mathcal{T}_{\omega -1})$.

\subsection{Initial pair: Constructive argument}

Let us construct the initial pair.
In other words, let us prove Proposition \ref{p.initiator}.

\begin{proof}[Proof of Proposition \ref{p.initiator}]
First, we define $u_{\mathrm{rot}} : S^1 \to S^1$ by
 $u_{\mathrm{rot}} (x) := x - R$. 
We define $u_I$ by performing perturbation to 
$u_{\mathrm{rot}}$ as follows:
\begin{enumerate}
\item For each $J(i)$ ($i = 0, \ldots, \xi-1$), 
we perturb $u_I[J(i)] = h$ ($h$ is the geometry eater renormalized to
$[0, R]$ see section \ref{s.geoe}). 
\item $u_I[J(\xi)] = \alpha_0$.
\item Except above, we do not do any perturbation.
\end{enumerate}

Note that since $\tau \geq  \xi + \omega +2$,
$J(\xi) = I (\tau -1 -\xi)$ is ``on the right of'' $I(\omega -1)$. 
Now we construct $M_I$.

\begin{itemize}
\item First, we define a region $M_t \subset [0, \omega R]$.
to be the $\omega$-template
$\mathcal{T} := \coprod T_i \subset [0, \omega]$ 
renormalized to this interval. 
\item Next, we define 
a region $M' \subset [1/q_1 - R(\xi +1), 1/q_1] = 
\cup_{i=0}^{\xi} J(i)$ 
by $M' := \cup_{i=1}^{\xi+1} u_{I}(I(0) \cap M_t)^i$.
Note that $J(\xi) = I (\tau -\xi -1)$.
\item On $[R\omega, R(\tau - \xi -1)] 
= \cup_{i=\omega}^{\tau - \xi -2 } I(i)$, we define 
$M_s \subset [R\omega, R(\tau - \xi -1)]$
such that for every $i = \omega, \ldots, \tau -\xi -2$,
we have $\pi(I(i) \cap M_s) = \mathcal{T}_{\omega -1}$.
Remember that $\tau \geq \xi + \omega +2$.
\item Finally, put $M_I = M' \cup M_t \cup M_s$.
\end{itemize}

Let us see that $(u_I, M_I)$ satisfies the hiding property.
In fact, for $I(i)$ in $[0, \omega R]$ it comes from the definition of 
template. The hiding property on 
$[1/q_1 - R(\xi +1), 1/q_1]$ comes from the fact that 
$M'$ is constructed by taking the images of $M_t \cap I(0)$ 
under $u_I$.
The hiding property on $[R\omega, R(\tau - \xi -1)] $  is easy.
 
We need to check the hiding property from 
$J(\xi) = I (\tau - \xi - 1 )$ to $J(\xi +1) = I (\tau - \xi - 2 )$.
To see this, let us trace $u_I^{i}(M_I \cap I(0))$. Note that 
$\pi(u_I^{i}(M_I \cap I(0))) = h^{i-1}(\mathcal{T}_0)$ 
for $i = 1, \ldots, \xi +1$
(more precisely, $T_0$ renormalized to $[0, R]$) .
Thus, in particular for $i = \xi +1$, 
we have $\pi(u_I^{\xi}(M_I \cap I(0))) = h^{\xi}(T_0)$.
Since we have $(u_I)[J(\xi)] = \alpha_0$,
we can see that $u_I$ has the hiding property 
from $J(\xi)$ to $J(\xi +1)$.

We can check the conditions on $\nu$ from 
the definition of the template.
Finally, we can check $\mu(u_I) < 1/\omega $ and 
its pasting property easily.
\end{proof}

\begin{rem}
For the better understanding of the proof, 
a figure like Figure \ref{fig.aliint} would be helpful.
The long segment is the $S^1$. It is divided into 
$I(i)$. The arrows shows the behavior of $u_I$.
It maps each $I(i)$ to $I(i-1)$.  
\begin{figure}
\centering
\includegraphics{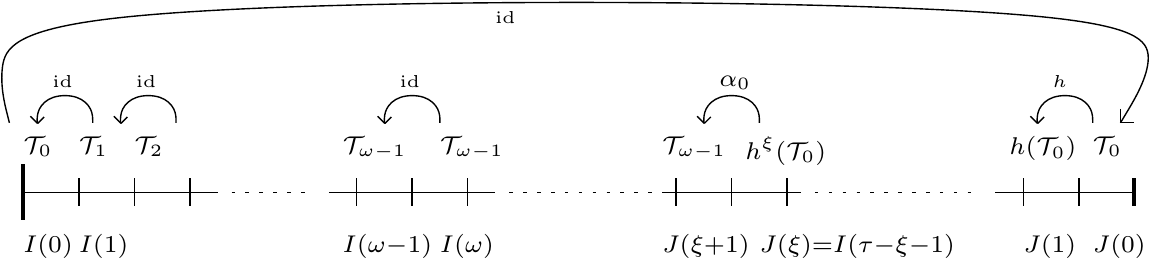}
 \caption{Alignment of $M_I$}.
\label{fig.aliint}
\end{figure}
\end{rem}

\subsection{Killing geometry}\label{s.kilgeo}

From $(u_I, M_I)$, we construct a $1/\omega$-equivalent 
sequence to decrease the geometric complexity 
in the template region.
First, we give the definition of our destination of this section.

\begin{defin}
Let $( u_p, M_p)$ be as follows;
\begin{itemize}
\item $u_p = x -R$.
\item $M_p$ is the (unique) region which satisfies the following:
for every $i$,  $\pi(M_p \cap I(i)) = T_{\omega -1}$.
\end{itemize}
\end{defin}

We prove the following.
\begin{lem}\label{l.georeal}
There exists a $1/\omega$-equivalent sequence of 
$R$-hiding pair $(u_j, M_j)$
$(j = 0,\ldots, \xi +1)$
which connects $( u_I, M_I)$ and $( u_p, M_p)$
satisfying the pasting property and $\mu (u_i) < 1/\omega$.
\end{lem}

\begin{proof}
First, we construct $(u_j, M_j)$ ($j = 0,\ldots, \xi $) as follows:
\begin{itemize}
\item The definition of $u_j$: $u_j$ is the perturbation 
of $u_{\mathrm{rot}}$ such that;
\begin{itemize}
\item On $J(i)$ for $i = 0,\ldots, \xi -1 - j$, 
we define $u_j[J(i)] = h$. 
\item On $J(\xi - j)$, 
we define $u_j[J(\xi - j)] = \alpha_0$. 
\item Otherwise we put $u_j = u_I = u_{\mathrm{rot}}$.
\end{itemize}
\item $M_j$ is the (unique) region which satisfies the following:
\begin{itemize}
\item $\pi(M_j \cap I(i)) = h^j(\mathcal{T}_i)$ for $i = 0, \ldots, \omega -2$.
\item On $J(i)$ for $i =0, \ldots, \xi -j$, we have 
 $\pi(M_j \cap J(i)) = h^{j+i}(\mathcal{T}_0)$.
\item For other $I(i)$, we have
 $\pi(M_j \cap I(i)) = \mathcal{T}_m$.
\end{itemize}
\end{itemize}
We can check that $(u_j, M_j)$ defines a hiding pair 
by the similar argument for the case of $M_r$. 
Note that $(u_0, M_0) = (u_I, M_I)$.

Now we construct the maps $\{U_j\}, \{V_j\}$ 
which connects $(u_j, M_j)$
and $(u_{j+1}, M_{j+1})$ as follows
($j=0,\ldots, \xi -1$):
\begin{itemize}
\item Definition of $U_j$. $U_j$ is the perturbation of $u_{\mathrm{rot}}$ satisfying the following:
\begin{itemize}
\item On $I(i)$ for $i = 0, \ldots, \omega -2$, 
we perturb $U_j[I(i)] = h$.
\item If $j \neq \xi - 1$,
then on $J(i)$ for $i = 0, \ldots , \xi -j -2$,  
we perturb $U_j[J(i)] = h^2$.
\item On $J(\xi -j -1)$,  
we perturb $U_j[J(\xi -j -1)] = \alpha_0 \circ h$.
\item On $J(\xi - j)$, 
we perturb $U_j[J(\xi - j)] = \alpha_0$.
\item Otherwise, no perturbation.
\end{itemize}
\item Definition of $V_j$. $V_j$ is the perturbation of 
$\mathrm{id}$ satisfying the following:

\begin{itemize}
\item On $I(i)$ for $i = 0, \ldots, \omega -2$, 
we perturb $V_j[I(i)] = h^{-1}$.
\item On $J(i)$ for $i = 0, \ldots, \xi -j -1$,  
we perturb $V_j[J(i)] = h^{-1}$.
\item On $J(\xi -j)$,   
we perturb $V_j[J(\xi -j)] = \beta_0$.
\item Otherwise, $V_j = \mathrm{id}$.
\end{itemize}
\end{itemize}

By definition, we can check that 
$U_i$, $V_i$ satisfy pasting property and 
$\mu(U_i), \mu(V_i) < 1/\omega$ using  
$\mu(h), \mu(\alpha_0) < \Omega$.
It is not difficult to check that these maps give the hiding property.
The best way is to draw the following type of diagrams.
\begin{center}
$ \hspace*{30pt}\xymatrix @R=.3pc{
\cdots  &
\ar[l]  h^j(\mathcal{T}_0) &
\ar[l]_{\mathrm{id}} \ar[lddd]_{h}h^j(\mathcal{T}_0) &
\ar[l]_{\mathrm{id}} \ar[lddd]_{h} h^j(\mathcal{T}_1) & 
\ar[l]|{\ldots\ldots}  \, h^j(\mathcal{T}_{\omega -2}) \, &
\ar[l]_{\mathrm{id}} \ar[lddd]_{h} \mathcal{T}_{\omega -1} &
\ar[l] \cdots &
\\
&&&&&&&\\
&&&&&&&\\
 \cdots  &
\ar[l] \ar[uuu]^{h^{-1}} h^{j+1}(\mathcal{T}_0) &
\ar[l]_{\mathrm{id}} \ar[uuu]^{h^{-1}} h^{j+1}(\mathcal{T}_0) &
\ar[l]_{\mathrm{id}} \ar[uuu]^{h^{-1}} h^{j+1}(\mathcal{T}_1) & 
\ar[l]|{\ldots\ldots } \ar[uuu]^{h^{-1}} h^{j+1}(\mathcal{T}_{\omega -2}) &
\ar[l]_{\quad \mathrm{id}} \ar[uuu]^{h^{-1}} \mathcal{T}_{\omega -1} &
\ar[l] \cdots &
\\
 & J(0) &I(0) & I(1) &   I(\omega -2) & I(\omega -1) &  & & 
}$
\end{center}

This diagram explains how $M_j$ and $M_{j+1}$ behaves 
under $u_j$, $u_{j+1}$, $U_j$ and $V_j$ near
$[0, R\omega] = \cup_{i=0}^{\omega -1} I(i)$.
In this diagram, the first row represents how $(M_j, u_j)$ looks like.
The arrows represent the behavior of $u_j$.
The second row represents how $(M_{j+1}, u_{j+1})$ looks like.
The third line explains the address of each region. 
By chasing each arrow, we can check the hiding property.

Next diagram explains the behavior near 
$J(0) \cup J(1) \cup J(2)$.
\begin{center}
$ \xymatrix @R=.3pc{
\cdots  &
\ar[l]  h^{j+2}(\mathcal{T}_0) &
\ar[l]_{h} \ar[lddd]_{h^2}h^{j+1}(\mathcal{T}_0) &
\ar[l]_{h} \ar[lddd]_{h^2} h^j(\mathcal{T}_0) & 
\ar[l] \cdots  \\
&&&& \\
&&&& \\
\cdots  &
\ar[l]  \ar[uuu]^{h^{-1}} h^{j+3}(\mathcal{T}_0) &
\ar[l]_{h} \ar[uuu]^{h^{-1}} h^{j+2}(\mathcal{T}_0) &
\ar[l]_{h} \ar[uuu]^{h^{-1}}  h^{j+1}(\mathcal{T}_0) & 
\ar[l]  \cdots  \\
  &  J(2) &J(1) & J(0) &  
}$
\end{center}
Again, the chase of arrows shows the hiding property.

The following one is around 
$J(\xi + j -2) \cup J(\xi + j -1) \cup J(\xi + j ) \cup J(\xi + j +1)$.
\begin{center}
$\xymatrix @R =0.3pc{
\cdots  &
\ar[l] \mathcal{T}_{\omega -1} &
\ar[l]_{\alpha_0} \ar[lddd]_{\alpha_0 } h^{\xi}(\mathcal{T}_0) &
\ar[l]_{h} \ar[lddd]_{(\alpha_0 \circ h)}h^{\xi -1}(\mathcal{T}_0) &
\ar[l]_{h} \ar[lddd]_{h^2} h^{\xi -2}(\mathcal{T}_0) & 
\ar[l]  \cdots  \\
&&&&& \\
&&&&& \\
\cdots  &
\ar[l] \ar[uuu]^{\mathrm{id}}  \mathcal{T}_{\omega -1} &
\ar[l]_{\mathrm{id}}  \ar[uuu]^{\beta_0} \mathcal{T}_{\omega -1} &
\ar[l]_{\alpha_0}  \ar[uuu]^{h^{-1}}  h^{\xi}(\mathcal{T}_0) &
\ar[l]_{h}  \ar[uuu]^{h^{-1}}  h^{\xi-1}(\mathcal{T}_0) &
\ar[l]  \cdots \\ 
 & J(\xi + j +1) &J(\xi + j ) & J(\xi + j -1) &  J(\xi + j -2) &
}$
\end{center}

The hiding property outside these regions are easy.
Finally, we see that $(u_{\xi}, M_{\xi})$ is $1/\omega$-equivalent
to $(u_{p}, M_p)$. Indeed, we just need to construct $U$ and $V$
which gives the $1/\omega$-equivalence.
It can be constructed as follows;

\begin{itemize}
\item Definition of $U$. $U$ is the perturbation of 
$u_{\mathrm{rot}}$ satisfying the following:
\begin{itemize}
\item On $I(0)$ and $J(0)$, for $i = 0, \ldots, \omega -2$, 
we perturb $U[I(0)] = U[J(0)] = \alpha_0$.
\item On $I(i)$ for $i = 1,\ldots \omega -2$, 
we perturb $U[I(i)] = \alpha_i$.
\item Otherwise, no perturbation.
\end{itemize}
\item Definition of $V$. $V_j$ is the perturbation of 
$\mathrm{id}$ satisfying the following:
\begin{itemize}
\item On $I(i)$ for $i = 0, \ldots, \omega -2$, 
we perturb $V_j[I(i)] = \beta_{i}$.
\item On $J(0)$,  
we perturb $V[J(0)] = \beta_0$.
\item Otherwise, $V_j = \mathrm{id}$.
\end{itemize}
\end{itemize}
It is easy to check that this gives the equivalence.

\end{proof}

\section{Quantum leap: The other case}\label{s.anoth}
In this section, we briefly discuss the construction of 
the quantum leap in the case $q_1 > q_2$. 
In fact, the construction itself is similar to the previous one, 
so we just see the difference between them.

Let $\omega$ be given. Then we fix $\lambda$ same value 
as was in the Section \ref{s.qua}. 
Then, take $p_1/q_1 >  p_2/q_2$ with $p_1/q_1 -  p_2/q_2 < \lambda /q_2$.
Take a piecewise linear map $g_1$ as follows.
\[
g_1(x) =
\begin{cases}
x + r_1          & (x \leq -\omega R ) \\
(1-1/\omega)x +r_1  & (-\omega R  \leq x \leq  0)   \\
x + r_2        & (x \geq 0)
\end{cases}
\]
We also take $f_1(x) := x -1$.

Let us see the action of $g_1$ and $f_1$. 
To see it, we choose the division $I(i)$ as was in the previous 
section modifying $I(i) := [i/q_2, (i+1)/q_2]$. 
We put $R = p_1/q_1 -  p_2/q_2$.
Then on  $\{ x \leq -\omega R\}$, we have 
$f_1(I(i)) = I(i - q_2)$ and $g_1(I(i)) = I(i+p_2) + R$.
Thus in this case the ``local action" of $g_1$ is translation of size $R$ 
to the right. In the previous case the direction was opposite.

Now, as was in the previous case we fix an initial pair $(u_0,  M_0)$.
In this case, since the direction is opposite, we need to 
choose everything to be the "mirror images" of the previous ones. 
Then, we take the $1/\omega$-equivalent sequences $(u_i, M_i)$ and 
$(U_i, V_i)$
so that the last one is the $1/q_1$ periodic.
The construction is completely same 
except the point that we take the mirror image of them. 

Then we modify the initial map gradually.
More precisely, we take the following grouping:
\begin{itemize}
\item $H\{ j \} := \{k \in \mathbb{Z} \mid j(p_2+q_2) \leq k < (j+1)(p_2 + q_2) \} $,
\item $H_l\{ j \} := \{k \in \mathbb{Z} \mid j(p_2+q_2) \leq k < j(p_2 + q_2) +p_2 \} $,
\item $H_r\{ j \} := \{k \in \mathbb{Z} \mid j(p_2+q_2) +p_2 \leq k < (j+1)(p_2 + q_2) \} $.
\end{itemize}
Then, give a perturbation respectively as was in Section \ref{c.qua1},
which finishes the proof.

\vskip 5mm
\begin{tabular}{ll}
Katsutoshi Shinohara & 
\\
\footnotesize{herrsinon@07.alumni.u-tokyo.ac.jp}
 &  \\
Institut de Math\'ematiques de Bourgogne,   CNRS - UMR 5584 
\\
Universit\'e de Bourgogne 9 av. A. Savary,& \\
21000  Dijon, France. &
\end{tabular}

\end{document}